\documentclass[10pt,reqno]{amsart}
\usepackage{amssymb,version,graphicx,fancybox,pifont}
\usepackage[numbers,sort,compress]{natbib}
\usepackage{multirow}
\usepackage{url,hyperref}
\usepackage{mathrsfs}
\usepackage{amsmath}
\usepackage{subfigure}
\usepackage{enumerate}
\usepackage{booktabs}
\usepackage[table, dvipsnames]{xcolor}
\usepackage{braket}
\usepackage{cleveref}
\usepackage{bm}
\catcode`\@=11
\@addtoreset{equation}{section}   

\@addtoreset{table}{section}   
\renewcommand\thefigure{\thesection.\@arabic\c@figure}
\@addtoreset{figure}{section}   

\renewcommand\thetable{\thesection.\@arabic\c@table}

 \newcommand{\new}{\newcommand*}
 \new{\rnew}{\renewcommand*}
  \new{\newe}{\newenvironment*}
 \new{\stl}{\setlength}
 \stl{\arraycolsep}{0.5mm}

\newtheorem{thm}{\bf Theorem}
\newtheorem{col}{Corollary}[section]
\newtheorem{prop}{Proposition}[section]

\newenvironment{theorem}{\begin{thm}} {\end{thm}}
\newtheorem{lmm}{\bf Lemma}

\newenvironment{lemma}{\begin{lmm}}{\end{lmm}}
\newtheorem{assumption}{Assumption}
\newenvironment{proposition}{\begin{prop}}{\end{prop}}
\newenvironment{corollary}{\begin{col}}{\end{col}}
\theoremstyle{remark}

\theoremstyle{definition}

\usepackage{multirow}

\newcommand {\bgeq}[1]{\begin{equation}\label{#1}}
\newcommand \edeq {\end{equation}}
\newcommand \bgth {\begin{theorem}\label}
\newcommand \edth {\end{theorem}}
\newcommand \bglm {\begin{lemma}\label}
\newcommand \edlm {\end{lemma}}
\newcommand {\bgar}[1]{\begin{array}{#1}}
\newcommand {\edar}{\end{array}}

\newtheorem{remark}{Remark}

\renewcommand \dfrac{\displaystyle\frac}

\usepackage{algorithm, algorithmic}

\let\originalparagraph\paragraph
\renewcommand{\paragraph}[1]{\originalparagraph{\textbf{#1}}}

\title[CARE-SAV for Gradient Flows]{CARE-SAV: A Conditioning-Aware Random-Feature Framework for Energy-Stable Simulation of Gradient Flows}
\author[B. Hu and Z. Li]{}

\thanks{$^\dag$Corresponding author  email:  zxli@shnu.edu.cn. \\ \indent The work was partially supported by the National Natural Science Foundation of China (No.12271366, 12571391, 11871043).}

\begin{document} \maketitle
\centerline{Bingcheng Hu, Zhaoxiang Li$^\dag$}
\centerline{Department of  Mathematics, Shanghai Normal University, Shanghai,
200234, P.R. China}


%

\begin{abstract}
Gradient-flow models are characterized by an intrinsic energy-dissipation structure, and faithfully preserving this structure at the discrete level is important for stable and reliable long-time simulation. To this end, we develop a Conditioning-Aware Representation Enhancement with Scalar Auxiliary Variable (CARE-SAV) framework, which constructs a compact spatial approximation space from flexible candidate features and evolves the gradient-flow dynamics directly within this space. The resulting fully discrete scheme preserves the discrete energy-dissipation law while providing a flexible alternative to conventional prescribed spatial discretizations. Rigorous analysis establishes the approximation capability, solvability, stability and convergence of the proposed method. Numerical experiments on representative gradient-flow problems demonstrate its accuracy, robustness and computational efficiency. We believe that CARE-SAV could provide a simple, flexible, and computationally efficient paradigm for structure-preserving discretization of gradient-flow problems.
\end{abstract}

\section{Introduction}

\subsection{Background and Motivation}
Gradient flows arise broadly in phase-field modeling, interfacial dynamics, materials science and many other dissipative systems \cite{SHENDecoupled2014,SHEN2015617,YANG20171116,Celledoni2018SISC}. Concretely, a state variable $\phi$ evolves in the direction of decreasing free energy $E$ according to
\begin{equation}
  \partial_t\phi=\mathcal G\frac{\delta E}{\delta\phi},
  \qquad
  \frac{\mathrm d}{\mathrm dt}E(\phi)=\left(\frac{\delta E}{\delta\phi},\mathcal G\frac{\delta E}{\delta\phi}\right)\le0,
  \label{eq:intro-gradient-flow}
\end{equation}
where $\mathcal G$ is a nonpositive mobility operator. Thus, the free energy is a Lyapunov functional and its decay describes irreversible relaxation toward equilibrium \cite{SHEN2018407,Tang2019SISC}. Reproducing this structure is essential for stable and physically consistent long-time simulation \cite{Shen-Phase-Field2010,GOMEZ20115310,Li2024SISC}, and is a central objective of structure-preserving computation.

Existing structure-preserving schemes for gradient flows are commonly formulated within prescribed spatial discretizations, including finite difference grids \cite{WiseFDM2009}, finite element spaces \cite{GaoFEM2018} and spectral approximation spaces \cite{LiShen2020SAV}. In these formulations, however, the spatial approximation space is prescribed a priori and remains tied to the underlying discretization. While such a setting provides a natural basis for variational formulation and numerical analysis, it also restricts the spatial representation to a predetermined class of approximation spaces. A natural question is therefore whether one can introduce a more flexible spatial representation without sacrificing the structural properties that underpin stable gradient-flow discretizations.

Achieving such flexibility is nontrivial, since approximation capability alone does not ensure compatibility with the structural requirements of gradient-flow discretizations. Motivated by this challenge, we develop a Conditioning-Aware Representation Enhancement with Scalar Auxiliary Variable (CARE-SAV) framework. This feature-generated Galerkin framework constructs a compact spatial approximation space from a redundant family of candidate functions while retaining the underlying energy-dissipation structure. It provides a flexible alternative to prescribed spatial discretizations without abandoning the structure-preserving principles central to simulation of gradient flows.

\subsection{Related Work}

\paragraph{Energy-stable schemes for gradient flows}
Energy-stable discretizations for gradient flows have been extensively studied over the past decades. Classical approaches include convex-splitting \cite{Eyre_1998,BaskaranConvex2013} and stabilized semi-implicit schemes \cite{Yang2009AllenCahn,LiSemiImplicitFourierSpectral,WangYu2018CahnHilliard}, while more recent developments such as invariant energy quadratization (IEQ) \cite{YangIEQ2018SISC,YangIEQ2017MMAMAS,YANG2017104JCP,ZhaoIEQ2017IJNME} and scalar auxiliary variable (SAV) \cite{SHEN2018407,SHENXU20182895,ChengMSAV2018SISC,HuangSAV2020SISC,HOU2019307} methods enable efficient linear schemes with unconditional energy stability. These approaches have established a powerful framework for preserving the dissipative structure of gradient flows. However, in fully discrete realizations, the spatial approximation is typically introduced through a separately prescribed discretization and therefore remains tied to a predetermined mesh, basis, or approximation space \cite{GongSAV2018SISC,HanBrylevYang2017}. This motivates the search for more flexible spatial representations that can retain the same structure-preserving properties.

\paragraph{Mesh-free and representation-based PDE solvers}
Mesh-free and representation-based PDE solvers offer an alternative route to spatial approximation by representing the solution in flexible families of functions rather than conventional mesh-dependent spaces \cite{Belytschko1994IJNME,chen2022JML,SIRIGNANO20181339}. Random-feature \cite{CHI2024116719,SHANG2023107518,SUN2024115830,chen2023}, extreme-learning \cite{CALABRO2021,DWIVEDI202096,DONG2021114129,QUAN2023697} and other related neural representations \cite{RAISSI2019686,EYu2018DeepRitz,KHARAZMI2021113547} have demonstrated the potential of flexible function classes for constructing accurate approximations to PDE solutions. Such approaches provide substantial freedom in the choice of spatial representation and can yield compact finite-dimensional approximations without relying on problem-specific basis functions. However, their design is typically driven by approximation accuracy, residual minimization or computational efficiency, while the discrete variational structure of the underlying PDE is not generally imposed as a primary requirement \cite{tang2026energy,wang2026pinns}. For gradient flows, this leaves open the question of how a feature-generated approximation space can be endowed with the structure required to support a rigorous energy-stable discretization.

\subsection{Main Contributions}

To address the gap, CARE-SAV is developed as a variational framework for gradient-flow computation that integrates flexible spatial approximation with structure-preserving discretization at the formulation level (see Fig.~\ref{fig:main_figure}). This perspective enables the spatial representation to remain adaptable while retaining the analytical properties required for stable long-time evolution. The main contributions of this work are summarized as follows:
\begin{itemize}
  \item A compact CARE space is constructed from flexible candidate features for spatial discretization.
  \item A fully discrete CARE-SAV scheme is developed to retain the dissipative structure of gradient flows.
  \item Solvability, stability, conservation and convergence properties are rigorously established.
  \item Representative experiments demonstrate the accuracy, robustness and efficiency of the proposed framework.
\end{itemize}

\begin{figure}[htbp]
  \centering
  \includegraphics[width=1\textwidth]{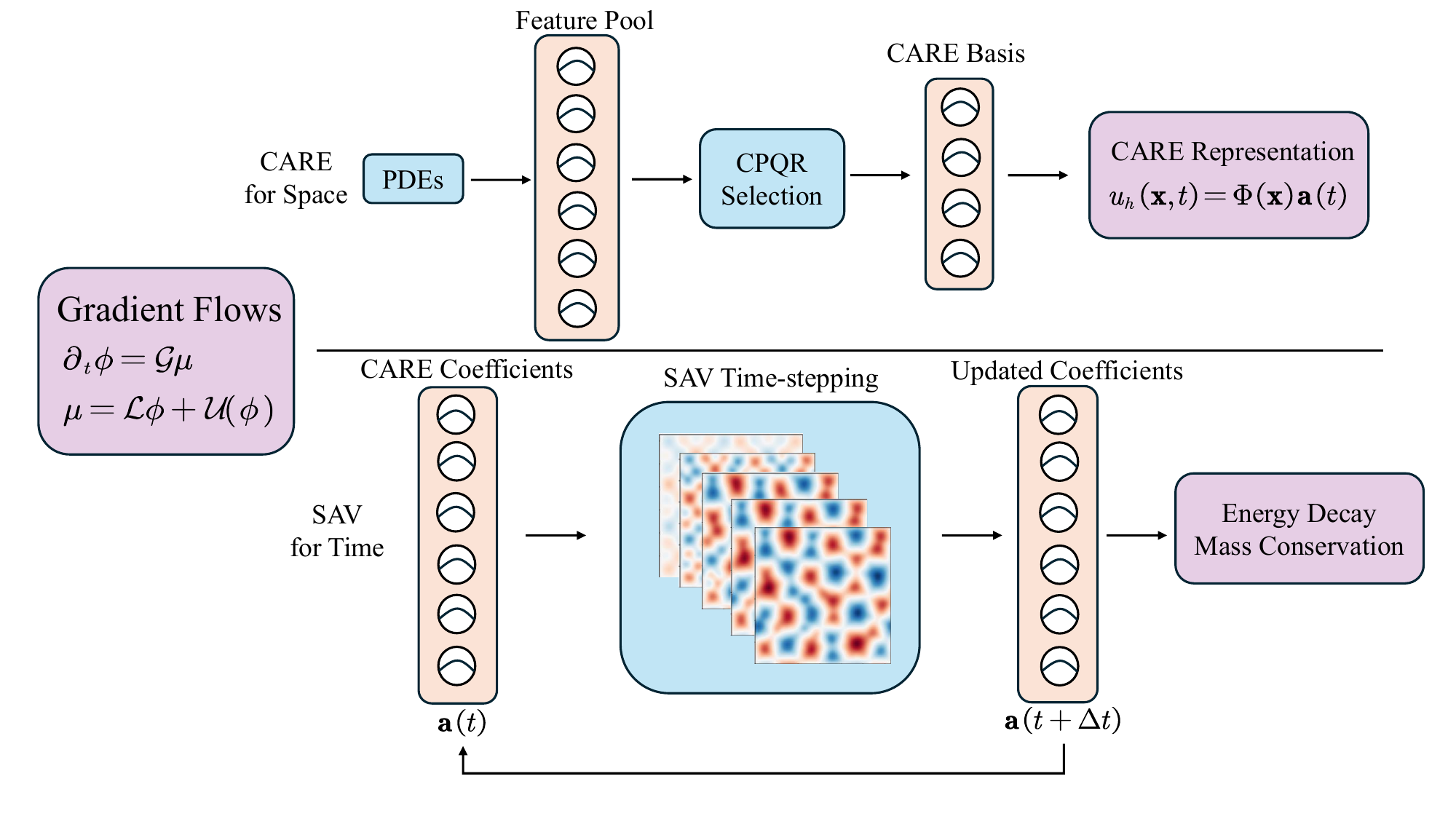}
  \vspace{-0.4in}
  \caption{CARE-SAV Framework}
  \label{fig:main_figure}
\end{figure}

\subsection{Organization}
This paper is organized as follows: Section 2 introduces the gradient-flow formulation and reviews the SAV framework. Section 3 presents the conditioning-aware representation enhancement and the construction of the CARE space. Section 4 develops the fully discrete CARE-SAV scheme and its efficient solution procedure. Section 5 establishes the solvability, energy stability, conditioning and error estimates of the proposed method. Section 6 reports numerical experiments for representative gradient-flow models and parameter study. Finally, Section 7 concludes the paper and discusses possible future directions.


\section{Gradient-Flow Formulation and the SAV Framework}\label{sec:sav-framework}
In this section, we introduce the general gradient-flow formulation underlying the class of problems considered in this work and briefly review the classical SAV-CN framework. We then derive the corresponding discrete energy law, which provides the structural foundation for the development and stability analysis of the proposed CARE-SAV method.

\subsection{Gradient-Flow Formulation}
We consider a general class of gradient-flow problems posed on a spatial domain $\Omega\subset\mathbb{R}^d$. Periodic or homogeneous Neumann boundary conditions are assumed so that the boundary terms arising from integration by parts vanish. Throughout this section, we denote the $L^2$ inner product and norm by $(u,v) = \int_\Omega u v \mathrm{d} \boldsymbol{x} $ and $\|u\| = (u,u)^{1/2}$ respectively.

The free-energy functional is assumed to admit the decomposition
\begin{equation}
  E(\phi) = \frac{1}{2}(\phi,\mathcal{L}\phi) + E_1(\phi),
\end{equation}
where $\mathcal{L}$ is a self-adjoint nonnegative linear operator and $E_1(\phi)$ contains the nonlinear contribution. We denote its variational derivative by $\mathcal{U} (\phi) = \frac{\delta E_1}{\delta \phi}$.

The corresponding gradient-flow system can be written as
\begin{equation}
  \partial_t \phi = \mathcal{G} \mu, \quad \mu = \mathcal{L}\phi + \mathcal{U}(\phi),
\end{equation}
where $\mathcal{G}$ is a nonpositive operator satisfying $(\mathcal{G}v,v) \leq 0$. Typical choices include $\mathcal{G} = -I$ for nonconserved gradient flows and $\mathcal{G} = \Delta$ for conserved gradient flows.

For the unforced gradient-flow system, taking the time derivative of the free energy gives
\begin{equation}
  \frac{\mathrm{d}}{\mathrm{d}t} E(\phi) = \left(\frac{\delta E}{\delta \phi},  \partial_t \phi \right) = (\mu, \mathcal{G} \mu) \le 0.
\end{equation}
Hence, in the absence of external forcing, the free energy decreases monotonically and serves as a Lyapunov functional of the system.

To construct a linear and energy-stable temporal discretization, we assume that the nonlinear energy $E_1(\phi)$ is bounded from below and choose a constant $C_0$ such that $E_1(\phi) + C_0 > 0$. This positivity condition permits the introduction of a scalar auxiliary variable and forms the basis of the SAV reformulation.

\subsection{The SAV-CN Scheme and Discrete Energy Law}
Under the assumption $E_1(\phi) + C_0 > 0$, we introduce the scalar auxiliary variable $r(t) = \sqrt{E_1(\phi) + C_0}$. Defining
$$b(\phi) := \frac{\mathcal{U}(\phi)}{\sqrt{E_1(\phi) + C_0}},$$
the chemical potential can be reformulated as
\begin{equation}
  \mu = \mathcal{L}\phi + r b(\phi) .
\end{equation}
This reformulation replaces the nonlinear energy contribution by a scalar variable while preserving equivalence with the original gradient-flow system at the continuous level.

The original gradient-flow problem can therefore be reformulated as the equivalent SAV system
\begin{equation}
  \partial_t \phi = \mathcal{G} \mu, \quad \mu = \mathcal{L}\phi + r b(\phi), \quad \frac{\mathrm{d}r}{\mathrm{d}t} = \frac{1}{2} (b(\phi), \partial_t \phi).
\end{equation}
At the continuous level, the relation $r(t) = \sqrt{E_1(\phi) + C_0}$ is preserved by this system.

Let $t_n=n\Delta t$, where $\Delta t$ denotes the time-step size. To achieve second-order temporal accuracy while retaining a discrete energy-dissipation structure, we employ the Crank-Nicolson discretization \cite{CrankNicolson1996}. For any sequence $\{v_n\}$, define the midpoint approximation and the second-order explicit extrapolation by
\begin{equation}
  \bar{v}^{n+1/2} = \frac{v^{n+1}+v^n}{2}, \quad \bar{v}^{n+1/2} = \frac{3v^n-v^{n-1}}{2}.
\end{equation}
Applying the Crank-Nicolson discretization to the linear terms and evaluating the nonlinear coefficient explicitly at $\bar{\phi} ^ {n+1/2}$, the classical SAV-CN scheme \cite{SHEN2018407} reads

\begin{equation}
\begin{aligned}
& \frac{\phi^{n+1}-\phi^n}{\Delta t} =\mathcal{G}\mu^{n+\frac12}, \quad \mu^{n+\frac12}=\mathcal{L}\phi^{n+\frac12} + r^{n+\frac12} b\left(\bar{\phi}^{\,n+\frac12}\right), \\
\quad
& \frac{r^{n+1}-r^n}{\Delta t} = \frac12
\left(b \left(\bar{\phi}^{\,n+\frac12}\right),
\frac{\phi^{n+1}-\phi^n}{\Delta t}
\right),
\end{aligned}
\end{equation}
where
\begin{equation}
\phi^{n+\frac12} = \frac{\phi^{n+1}+\phi^n}{2}, \quad
r^{n+\frac12} = \frac{r^{n+1}+r^n}{2}.
\end{equation}
The midpoint structure of the Crank-Nicolson method is compatible with the SAV reformulation, while the explicit treatment of $b(\phi)$ yields a linear scheme in which only linear systems need to be solved at each time step.

Associated with the SAV-CN scheme, we define the discrete modified energy by

\begin{equation}
  \widetilde{E}(\phi^n,r^n) = \frac12(\phi^n,\mathcal{L}\phi^n) + (r^n)^2 - C_0.
\end{equation}
The scheme satisfies the discrete energy law
\begin{equation}
  \widetilde{E}^{n+1} - \widetilde{E}^n = \Delta t (\mu^{n+\frac12}, \mathcal{G}\mu^{n+\frac12}) \le 0.
\end{equation}
Therefore, the modified energy is nonincreasing for any time-step size, which establishes the unconditional energy stability of the classical SAV-CN scheme. Consequently, the modified energy is nonincreasing for any time-step size, which establishes the unconditional energy stability of the classical SAV-CN scheme.

The classical SAV-CN framework provides an energy-stable temporal discretization, while the construction of an accurate and well-conditioned spatial approximation remains an independent issue. A direct random-feature representation may contain strongly correlated or numerically redundant basis functions, leading to ill-conditioned discrete systems and degraded accuracy. To address this difficulty, we introduce a conditioning-aware representation enhancement strategy that constructs a compact and numerically stable approximation space from an initially overcomplete set of candidate features.

\section{Conditioning-Aware Representation Enhancement}
In this section, we develop the conditioning-aware representation enhancement strategy CARE. First we introduce the random-feature trial space, then construct a numerically stable basis by identifying and removing redundant feature directions, and finally formulate the Galerkin projection in the resulting CARE space.

\subsection{Random-Feature Trial Space}
Let $\Omega \subset \mathbb{R}^d$ be the computational domain. We first generate an overcomplete collection of $M$ boundary-compatible Gaussian candidate features. For problems with natural boundary conditions, we use the usual Euclidean Gaussian radial basis functions
\begin{equation}
  \varphi_j(\boldsymbol{x}) = \exp\left(-\frac{\|\boldsymbol{x}-\boldsymbol{c}_j\|^2}{2\sigma_j^2}\right), \quad j=1,\cdots,M.
  \label{eq:Euclidean-Gaussian-feature}
\end{equation}
For a periodic problem on the rectangular cell $\Omega=\prod_{\ell=1}^d[a_\ell,a_\ell+L_\ell)$, the Euclidean Gaussian is replaced by its periodized counterpart
\begin{equation}
  \varphi_j^{\mathrm{per}}(\boldsymbol{x})
  =\sum_{\boldsymbol{k}\in\mathbb Z^d}
  \exp\left(-\frac{\|\boldsymbol{x}-\boldsymbol{c}_j+\boldsymbol{k}\odot\boldsymbol L\|^2}{2\sigma_j^2}\right),
  \quad j=1,\cdots,M,
  \label{eq:periodized-Gaussian-feature}
\end{equation}
where $\boldsymbol L=(L_1,\cdots,L_d)$ and $\odot$ denotes componentwise multiplication. Here $\boldsymbol{c}_j \in \Omega$ and $\sigma_j > 0$ denote the randomly selected center and width of the $j$-th feature respectively. The centers are sampled over the computational domain, while the widths are selected from a prescribed interval $[\sigma_{\min}, \sigma_{\max}]$. Thus, for periodic problems, the Gaussian candidate features are periodized so that the resulting CARE space is compatible with the prescribed periodic boundary conditions.

The candidate functions span the initial random-feature trial space
\begin{equation}
  V_{M}^{\text{RF}} = \mathrm{span}\{\varphi_1,\cdots,\varphi_M\}.
\end{equation}
In the periodic case, we use $\varphi_j$ to denote $\varphi_j^{\mathrm{per}}$. If $X_{\mathrm{per}}$ denotes the periodic energy space associated with the model, then $V_M^{\mathrm{RF}}\subset X_{\mathrm{per}}$. Moreover, since CARE selects only linear combinations of the candidate features, we also have $V_K^{\mathrm{CARE}}\subset X_{\mathrm{per}}$.
Here, $M$ denotes the number of candidate features. Since the features are generated randomly, $V_{M}^{\text{RF}}$ may contain strongly correlated or numerically redundant directions, and its effective numerical rank can be considerably smaller than $M$. The approximate solution is represented in the random-feature trial space as
\begin{equation}
  \phi(\boldsymbol{x},t) = \sum_{j=1}^{M} a_j(t) \varphi_j(\boldsymbol{x}) = \boldsymbol{\varphi}(\boldsymbol{x})^{T} \boldsymbol{a}(t),
\end{equation}
where $\boldsymbol{\varphi}(\boldsymbol{x}) = [\varphi_1(\boldsymbol{x}),\cdots,\varphi_M(\boldsymbol{x})]^{T}$ and $\boldsymbol{a}(t) = [a_1(t),\cdots,a_M(t)]^{T}$. The feature centers and widths are fixed after initialization, while only the coefficient vector $\boldsymbol{a}(t)$ evolves in time.

Let $\{\boldsymbol{x}_q,w_q\}_{q=1}^{N}$ denote the spatial quadrature points and weights. The candidate features are evaluated at these points to form the feature matrix $\Phi \in \mathbb{R}^{N \times M}$, whose entries are $\Phi_{qj} = \varphi_j(\boldsymbol{x}_q)$. For a spatial differential operator $\mathcal{D}$, the corresponding operator-evaluation matrix is defined by $(\mathcal{D}\Phi)_{qj} = \mathcal{D}\varphi_j(\boldsymbol{x}_q)$. These matrices provide discrete representations of the trial functions and their spatial derivatives and will be used to construct the conditioning-aware CARE basis.

Since the candidate feature set may contain strongly correlated or numerically redundant directions, the matrices introduced above can be severely ill-conditioned. To remedy this issue, a conditioning-aware basis is constructed in the following subsection by retaining only the numerically independent feature directions.

\subsection{Conditioning-Aware Basis Construction}
Let $\widehat{\Phi}  = W^{1/2} \Phi$ denote the weighted feature matrix. The CARE construction begins by identifying the numerical column space of $\widehat{\Phi}$ with respect to the discrete $L^2(\Omega)$ inner product. Since the candidate feature matrix may be rank deficient or severely ill-conditioned, retaining all $M$ feature directions is unnecessary and numerically undesirable. We therefore replace $V_{M}^{\text{RF}}$ by a reduced subspace of dimension $K \leq M$, spanned by numerically independent directions. The quadrature weighting ensures that this reduction is consistent with the discrete $L^2(\Omega)$ geometry.

To determine the numerical rank of $\widehat{\Phi}$, we employ a column-pivoted rank-revealing QR (CPQR) factorization
$$
  \widehat{\Phi} \mathsf{P} = Q \mathsf{R},
$$
where $\mathsf{P} \in \mathbb{R}^{M \times M}$ is a permutation matrix, $Q \in \mathbb{R}^{N \times M}$ has orthonormal columns, and $\mathsf{R} \in \mathbb{R}^{M \times M}$ is upper triangular. The column permutation orders the candidate features according to their incremental contribution to the numerical column space. Consequently, the leading columns correspond to the dominant independent directions, whereas nearly dependent or redundant features are relegated to the trailing part of the factorization.

The effective dimension $K$ is determined from the diagonal entries of $\mathsf{R}$ using a prescribed tolerance $\tau_{QR}>0$. Specifically, we set
\begin{equation}
  K=\max\left\{k \mid \frac{\left|\mathsf{R}_{kk}\right|}{\left|\mathsf{R}_{11}\right|}\geq\tau_{\mathrm{QR}} \right\}.
\end{equation}
The leading $K$ pivoted columns are therefore retained, while the remaining directions are regarded as numerically dependent and discarded. The effective dimension $K$ is determined adaptively by the rank-revealing procedure. A larger value of $\tau_{QR}$ yields a more aggressive reduction, whereas a smaller value retains a larger portion of the candidate feature space.

Let $\mathsf{P}_K \in \mathbb{R}^{M \times K}$ contain the first $K$ columns of $\mathsf{P}$, and let $Q_K \in \mathbb{R}^{N \times K}$ and $\mathsf{R}_K \in \mathbb{R}^{K \times K}$ denote the corresponding leading factors, so that
$$
  W^{1/2} \Phi \mathsf{P}_K = Q_K \mathsf{R}_K.
$$
Since $\mathsf{R}_K$ is nonsingular, define the transformation matrix $\mathsf T = \mathsf{P}_K \mathsf{R}_K^{-1}$ and the CARE basis matrix $\Psi = \Phi \mathsf T$. It then follows that
$$
  W^{1/2} \Psi = W^{1/2} \Phi \mathsf T = Q_K, \quad \Psi ^{T} W \Psi = I_K.
$$
Hence, the resulting CARE basis is orthonormal with respect to the discrete $L^2(\Omega)$ inner product. The resulting conditioning-aware approximation space is defined by
\begin{equation}
  V_K^{\text{CARE}} = \mathrm{span}\{\psi_1,\cdots,\psi_K\} \subset V_M^{\text{RF}},
\end{equation}
where $\{\psi_j\}_{j=1}^K$ are the basis functions represented by the columns of $\Psi$. The CARE basis is constructed once prior to time integration and is subsequently fixed throughout the simulation. The resulting reduced space has an identity discrete mass matrix and is employed as the trial space in the Galerkin discretization.

\subsection{Galerkin Projection in the CARE Space}
We seek approximations $\phi_K , \mu_K \in V_K^{\text{CARE}}$ of the form
$$
  \phi _K(\boldsymbol{x},t) = \sum_{j=1}^{K} c_j(t) \psi_j(\boldsymbol{x}) = \boldsymbol{\psi}(\boldsymbol{x})^{T} \boldsymbol{c}(t), \quad \mu_K(\boldsymbol{x},t) = \sum_{j=1}^{K} d_j(t) \psi_j(\boldsymbol{x}) = \boldsymbol{\psi}(\boldsymbol{x})^{T} \boldsymbol{d}(t),
$$
where $\boldsymbol{c}(t) = [c_1(t), \cdots, c_K(t)]^{T}$ and $\boldsymbol{d}(t) = [d_1(t), \cdots, d_K(t)]^{T}$ are the coefficient vectors. The Galerkin formulation is obtained by imposing orthogonality of the residuals to $V_K^{\text{CARE}}$. Using the quadrature-based inner product
$$
  (u,v)_Q = \sum_{q=1}^{N} w_q u(\boldsymbol{x}_q) v(\boldsymbol{x}_q),
$$
the CARE Galerkin formulation seeks $\phi _ K ,\mu _K \in V_K^{\text{CARE}}$ such that $ \forall v_K \in V_K^{\text{CARE}}$, we have
\begin{equation}
  (\partial_t\phi_K,v_K)_Q=(\mathcal{G}\mu_K,v_K)_Q,\quad(\mu_K,v_K)_Q =(\mathcal{L}\phi_K,v_K)_Q + (\mathcal{U}(\phi_K),v_K)_Q.
\end{equation}
The discrete mass matrix is the identity because the CARE basis is orthonormal under $(\cdot,\cdot)_Q$.

Define the reduced operator matrices and nonlinear vector by
$$
  (\boldsymbol{G}_K)_{ij}=(\mathcal{G}\psi_j,\psi_i)_Q,\quad(\boldsymbol{L}_K)_{ij}=(\mathcal{L}\psi_j,\psi_i)_Q,\quad[\boldsymbol{U}_K(\boldsymbol{c})]_i=(\mathcal{U}(\phi_K),\psi_i)_Q.
$$
Using the discrete orthonormality of the CARE basis, the Galerkin system takes the coefficient form
$$
  \dot{\boldsymbol{c}}=\boldsymbol{G}_K\boldsymbol{d},\quad\boldsymbol{d}=\boldsymbol{L}_K\boldsymbol{c}+\boldsymbol{U}_K(\boldsymbol{c}).
$$
This reduced system provides the spatial semidiscretization employed in the CARE-SAV formulation.

\section{The CARE-SAV Method}
\label{sec:CARE-SAV}
In this section, we combine the CARE Galerkin approximation with the SAV framework introduced in Section~\ref{sec:sav-framework}. We first derive the reduced spatially semidiscrete system, then present its fully discrete Crank-Nicolson formulation, and finally describe an efficient solution procedure.

\subsection{CARE-SAV Spatial Semidiscretization}
Applying the SAV reformulation to the CARE Galerkin system derived in the preceding section gives the spatially semidiscrete CARE-SAV formulation. For a coefficient vector $\boldsymbol{c} \in \mathbb{R}^K$, define the projected SAV vector $\boldsymbol{b}_K(\boldsymbol{c}) \in \mathbb{R}^K$ by
\begin{equation}
  \left[\boldsymbol b_K(\boldsymbol c)\right]_i =\left(\frac{\mathcal U(\phi_K)}{\sqrt{E_{1,Q}(\phi_K)+C_0}},\psi_i\right)_Q,\quad i=1,\cdots, K,
\end{equation}
where $E_{1,Q}(\phi_K)$ denotes the quadrature approximation of the nonlinear energy and $\phi_K$ is determined by $\boldsymbol{c}$. The resulting coefficient system is
\begin{equation}
  \dot{\boldsymbol {c}}=\boldsymbol{G}_K \boldsymbol{d}, \quad
  \boldsymbol{d} = \boldsymbol{L}_K\boldsymbol{c}+r\boldsymbol{b}_K(\boldsymbol{c}), \quad
  \dot{r} =  \frac{1}{2} \boldsymbol{b}_K(\boldsymbol{c})^{T}\dot{\boldsymbol{c}}.
\end{equation}
The initial condition is defined by the discrete $L^2$ projection of $\phi_0$ onto the CARE space. Owing to the discrete orthonormality of the CARE basis, the initial coefficients and auxiliary variable are given by
\begin{equation}
  c_i(0)=(\phi_0,\psi_i)_Q,
  \quad
  i=1,\cdots,K,
  \quad
  r(0) =\sqrt{E_{1,Q}\left(\phi_K(0)\right)+C_0}.
\end{equation}
The corresponding semidiscrete modified energy is defined as
\begin{equation}
\widetilde E_K(t)=\frac{1}{2}
\boldsymbol c(t)^{T}\boldsymbol{L}_K\boldsymbol c(t) +
r(t)^2-C_0.
\end{equation}
Using the symmetry of $\boldsymbol{L}_K$ and the coefficient system above, one obtains
\begin{equation}
  \frac{\mathrm{d}}{\mathrm{d}t}\widetilde E_K(t)= \boldsymbol{d}(t)^{T} \boldsymbol{G}_K\boldsymbol{d}(t) \leq 0.
\end{equation}
Thus, the CARE Galerkin projection retains the energy-dissipation structure of the SAV formulation at the spatially semidiscrete level.

\subsection{Fully Discrete CARE-SAV Scheme}
Using the Crank-Nicolson discretization and the explicit extrapolation introduced in Section~\ref{sec:sav-framework}, we define the coefficient-level midpoint and extrapolated values by
\begin{equation}
  \boldsymbol{c}^{n+\frac{1}{2}}=\frac{\boldsymbol{c}^{n+1}+\boldsymbol{c}^{n}}{2},\quad
  r^{n+\frac{1}{2}} = \frac{r^{n+1}+r^{n}}{2}, \quad
  \overline{\boldsymbol{c}}^{\,n+\frac12}=\frac{3\boldsymbol{c}^{n}-\boldsymbol{c}^{n-1}}{2}.
\end{equation}
The projected nonlinear vector is evaluated explicitly as
\begin{equation}
  \boldsymbol{b}^{\,n+\frac{1}{2}} = \boldsymbol{b}_K\left( \overline{\boldsymbol{c}}^{\,n+\frac{1}{2}} \right).
\end{equation}
For $n \geq 1$, the fully discrete CARE-SAV scheme seeks $\boldsymbol{c}^{n+1}\in \mathbb{R}^K$, $\boldsymbol{d}^{n+\frac{1}{2}}\in \mathbb{R}^K$, and $r^{n+1}\in \mathbb{R}$ satisfying
\begin{equation}
\begin{aligned}
  &\frac{\boldsymbol{c}^{n+1}-\boldsymbol{c}^{n}}{\Delta t}=\boldsymbol{G}_K\boldsymbol{d}^{n+\frac12},
  \\
  &\boldsymbol{d}^{n+\frac12}=\boldsymbol{L}_K\boldsymbol{c}^{n+\frac12}+r^{n+\frac12}\boldsymbol{b}^{\,n+\frac12},
  \\
  &\frac{r^{n+1}-r^{n}}{\Delta t}=\frac{1}{2}\left(\boldsymbol{b}^{\,n+\frac12}\right)^{T}\frac{\boldsymbol{c}^{n+1}-\boldsymbol{c}^{n}}{\Delta t}.
\end{aligned}
\label{eq:fully_discrete_care_sav}
\end{equation}
Since $b^{n+1/2}$ is evaluated explicitly, the resulting system is linear with respect to the unknown variables at $t_{n+1}$.

Specifically, for the first time step, we set
\begin{equation}
  \overline{\boldsymbol{c}}^{\,\frac12}=\boldsymbol{c}^{0},\quad
  \boldsymbol{b}^{\,\frac12} =\boldsymbol{b}_K(\boldsymbol{c}^{0}),
\end{equation}
and apply the same coupled update with $n=0$. The discrete modified energy associated with the scheme is defined by
\begin{equation}
  \widetilde{E}_K^{\,n}=\frac{1}{2}(\boldsymbol{c}^{n})^{T} \boldsymbol{L}_K \boldsymbol{c}^{n} + (r^{n})^2 -C_0.
\end{equation}
The fully discrete scheme retains the SAV energy structure. Its unique solvability and unconditional energy stability are established in Section~\ref{sec:proof}.

\subsection{Efficient Solution Procedure}
The coupled CARE-SAV scheme can be implemented without solving a fully coupled block system. From the auxiliary-variable equation in Eq.~\eqref{eq:fully_discrete_care_sav}, we obtain
\begin{equation}
  r^{n+\frac12} = r^n + \frac14 \left(\boldsymbol b^{\,n+\frac12}\right)^{T} \left(\boldsymbol c^{n+1}-\boldsymbol c^n\right).
\label{eq:r_half_elimination}
\end{equation}
Introduce the scalar quantity
\begin{equation}
  \alpha^{n+\frac{1}{2}} =\left(\boldsymbol {b}^{\,n+\frac{1}{2}}\right)^{T}\left(\boldsymbol c^{n+1}-\boldsymbol c^n\right).
\end{equation}
Substituting Eq.~\eqref{eq:r_half_elimination} into the first two equations of Eq.~\eqref{eq:fully_discrete_care_sav} gives
\begin{equation}
  \boldsymbol{A}_K\boldsymbol{c}^{n+1} = \boldsymbol{g}^n + \frac{\Delta t}{4} \alpha^{n+\frac12} \boldsymbol{G}_K\boldsymbol b^{\,n+\frac12},
\label{eq:reduced_linear_system}
\end{equation}
where
\begin{equation}
\begin{aligned}
&\boldsymbol{A}_K = \boldsymbol{I}_K - \frac{\Delta t}{2}\boldsymbol{G}_K\boldsymbol{L}_K,
\\
&\boldsymbol g^n = \left( \boldsymbol{I}_K + \frac{\Delta t}{2}\boldsymbol{G}_K\boldsymbol{L}_K \right)\boldsymbol c^n + \Delta t\,r^n \boldsymbol{G}_K\boldsymbol b^{\,n+\frac12}.
\end{aligned}
\label{eq:efficient_solver_definitions}
\end{equation}
We first compute $\boldsymbol{c}_{*}^{n+1}$and $\boldsymbol{q}^{n+1/2}$ from two linear systems with the same coefficient matrix below
\begin{equation}
  \boldsymbol{A}_K\boldsymbol c_*^{\,n+1} = \boldsymbol g^n, \quad
  \boldsymbol{A}_K\boldsymbol q^{\,n+\frac12}=\boldsymbol{G}_K\boldsymbol b^{\,n+\frac12}.
\label{eq:two_linear_solves}
\end{equation}
Eq.~\eqref{eq:reduced_linear_system} then implies
\begin{equation}
  \boldsymbol c^{n+1} = \boldsymbol c_*^{\,n+1} + \frac{\Delta t}{4} \alpha^{n+\frac12}\boldsymbol q^{\,n+\frac12}.
\label{eq:c_solution_decomposition}
\end{equation}
Taking the inner product of Eq.~\eqref{eq:c_solution_decomposition} with $\boldsymbol b^{\,n+1/2}$ determines the scalar coupling coefficient:
\begin{equation}
  \alpha^{n+\frac12} = \frac{\left(\boldsymbol b^{\,n+\frac12}\right)^{T}\left(\boldsymbol c_*^{\,n+1}-\boldsymbol c^n\right)}{1-\dfrac{\Delta t}{4}\left(\boldsymbol b^{\,n+\frac12}\right)^{T}\boldsymbol q^{\,n+\frac12}}.
\label{eq:scalar_alpha}
\end{equation}
The new solution and auxiliary variable are subsequently recovered from
\begin{equation}
\begin{aligned}
  &\boldsymbol c^{n+1} = \boldsymbol c_*^{\,n+1} + \frac{\Delta t}{4} \alpha^{n+\frac12}\boldsymbol q^{\,n+\frac12},\\
  &r^{n+1} = r^n+\frac12\alpha^{n+\frac12}, \\
  &\boldsymbol d^{n+\frac12} = \boldsymbol{L}_K\boldsymbol c^{n+\frac12} + r^{n+\frac12}\boldsymbol b^{\,n+\frac12}.
\end{aligned}
\label{eq:care_sav_solution_recovery}
\end{equation}
For a fixed time step, $\boldsymbol{A}_K$ is independent of $n$. Its factorization can therefore be computed once and reused throughout the simulation. Each time step requires two linear solves with the same $K\times K$ matrix, followed by matrix-vector and scalar operations.

The complete CARE-SAV procedure is summarized in Alg.~\ref{alg:care_sav}. The algorithm requires only two linear solves per time step, and the coefficient matrix is independent of time. The factorization of $\boldsymbol{A}_K$ can therefore be computed once and reused throughout the simulation, yielding a highly efficient implementation.

\begin{algorithm}[t]
\caption{CARE-SAV time integration}
\label{alg:care_sav}
\begin{algorithmic}[1]
\REQUIRE CARE basis, reduced matrices $\boldsymbol{G}_K$ and $\boldsymbol{L}_K$,
         time step $\Delta t$, initial values $\boldsymbol c^0$ and $r^0$,
         and number of time steps $N_t$
\ENSURE $\{\boldsymbol c^n,r^n\}_{n=1}^{N_t}$

\STATE Form $\boldsymbol{A}_K=\boldsymbol{I}_K-\frac{\Delta t}{2}\boldsymbol{G}_K\boldsymbol{L}_K$
\STATE Compute and store a factorization of $\boldsymbol{A}_K$

\FOR{$n=0,\cdots,N_t-1$}
    \IF{$n=0$}
        \STATE $\overline{\boldsymbol c}^{\,1/2}\gets\boldsymbol c^0$
    \ELSE
        \STATE $\overline{\boldsymbol c}^{\,n+1/2}\gets(3\boldsymbol c^n-\boldsymbol c^{n-1})/2$
    \ENDIF

    \STATE $\boldsymbol b^{\,n+1/2}\gets\boldsymbol b_K(\overline{\boldsymbol c}^{\,n+1/2})$

    \STATE $\boldsymbol g^n\gets\left(\boldsymbol{I}_K+\frac{\Delta t}{2}\boldsymbol{G}_K\boldsymbol{L}_K\right)\boldsymbol c^n
    +\Delta t\,r^n\boldsymbol{G}_K\boldsymbol b^{\,n+1/2}$

    \STATE Solve
    $\boldsymbol{A}_K\boldsymbol c_*^{\,n+1}=\boldsymbol g^n$

    \STATE Solve
    $\boldsymbol{A}_K\boldsymbol q^{\,n+1/2}
    =
    \boldsymbol{G}_K\boldsymbol b^{\,n+1/2}$

    \STATE $\displaystyle
    \alpha^{n+1/2}
    \gets
    \frac{
    (\boldsymbol b^{\,n+1/2})^{T}
    (\boldsymbol c_*^{\,n+1}-\boldsymbol c^n)
    }{
    1-
    \frac{\Delta t}{4}
    (\boldsymbol b^{\,n+1/2})^{T}
    \boldsymbol q^{\,n+1/2}
    }$

    \STATE $\displaystyle
    \boldsymbol c^{n+1}
    \gets
    \boldsymbol c_*^{\,n+1}
    +
    \frac{\Delta t}{4}
    \alpha^{n+1/2}
    \boldsymbol q^{\,n+1/2}$

    \STATE $\displaystyle
    r^{n+1}
    \gets
    r^n+\frac12\alpha^{n+1/2}$
\ENDFOR
\end{algorithmic}
\end{algorithm}

\section{Stability, Error and Conditioning Analysis}
\label{sec:proof}
This section presents the main theoretical results for the CARE-SAV method, including unique solvability, unconditional energy stability, and fully discrete convergence. Mass conservation for the Cahn-Hilliard case is also established. Detailed proofs and auxiliary results are provided in the Appendix.

\subsection{Unique Solvability}
We first establish the well-posedness of the fully discrete CARE-SAV scheme. By construction, the CARE basis is discretely orthonormal, we have
\begin{equation}
  \Psi^{T} \boldsymbol{W} \Psi=\boldsymbol{I}_K.
\end{equation}

\begin{assumption}
\label{as1}
  We assume that the reduced operators satisfy
  \begin{equation}
    \boldsymbol{L}_K=\boldsymbol{L}_K^{T}\succeq0,\quad\boldsymbol z^{T}\boldsymbol{G}_K\boldsymbol z\le0 \quad \forall \boldsymbol z\in\mathbb R^K.
  \end{equation}
\end{assumption}
These conditions are the discrete counterparts of the nonnegativity of the linear energy operator and the dissipativity of the mobility operator in the underlying gradient flow. They are naturally inherited under a structure-preserving projection and assembly of the reduced operators, and can be enforced or verified directly at the matrix level.

\begin{theorem}[Unconditional Unique Solvability]
\label{thm:uus}
  Let $\boldsymbol b^{n+1/2}=\boldsymbol b_K(\overline{\boldsymbol c}^{\,n+1/2})$ be fixed by explicit extrapolation. Under Assumption~\ref{as1}, the fully discrete CARE-SAV scheme has a unique solution
  $$
    (\boldsymbol c^{n+1},\boldsymbol d^{n+1/2},r^{n+1}) \in\mathbb R^K\times\mathbb R^K\times\mathbb R
  $$
  for every $\Delta t>0$.
\end{theorem}
\begin{proof}
  The fully discrete scheme is equivalent to the block system
  \begin{equation}
    \label{eq:square_system_pf1}
    \begin{bmatrix}
  \boldsymbol{I}_K/\Delta t&-\boldsymbol{G}_K&0\\
  -\boldsymbol{L}_K/2&\boldsymbol{I}_K&-\boldsymbol b^{n+1/2}/2\\
  -(\boldsymbol b^{n+1/2})^{T}/2&0&1
  \end{bmatrix}
  \begin{bmatrix}
  \boldsymbol c^{n+1}\\
  \boldsymbol d^{n+1/2}\\
  r^{n+1}
  \end{bmatrix}
  =
  \begin{bmatrix}
  \boldsymbol c^n/\Delta t\\
  \boldsymbol{L}_K\boldsymbol c^n/2+\boldsymbol b^{n+1/2}r^n/2\\
  r^n-(\boldsymbol b^{n+1/2})^{T}\boldsymbol c^n/2
  \end{bmatrix}.
  \end{equation}
  It is enough to prove that the associated homogeneous system has only the zero solution. Let $(\delta\boldsymbol c,\boldsymbol d,\delta r)$ satisfy
  \begin{equation}
  \label{eq:matrix_pf1}
    \frac{\delta\boldsymbol c}{\Delta t}=\boldsymbol{G}_K\boldsymbol d,
    \quad
    \boldsymbol d=\frac12\boldsymbol{L}_K\delta\boldsymbol c +\frac12\delta r\,\boldsymbol b,
    \quad
    \delta r=\frac12\boldsymbol b^{T}\delta\boldsymbol c,
  \end{equation}
  where $\boldsymbol b=\boldsymbol b^{n+1/2}$. Taking the inner product of the second equation with $\delta\boldsymbol c$ and using the third equation of Eq.~\eqref{eq:matrix_pf1} gives
  \begin{equation}
  \label{eq:uus_pf1}
    \delta\boldsymbol c^{T}\boldsymbol d =\frac12\delta\boldsymbol c^{T}\boldsymbol{L}_K\delta\boldsymbol c +(\delta r)^2 \ge 0.
  \end{equation}
  On the other hand, the first equation of Eq.~\eqref{eq:matrix_pf1} gives
  \begin{equation}
    \delta\boldsymbol c^{T}\boldsymbol d =\Delta t\,\boldsymbol d^{T}\boldsymbol{G}_K\boldsymbol d\le0.
  \end{equation}
  Consequently all nonnegative terms in Eq.~\eqref{eq:uus_pf1} vanish. Since $\boldsymbol{L}_K\succeq0$, we have $\delta\boldsymbol c^{T}\boldsymbol{L}_K\delta\boldsymbol c=0$, so that $\boldsymbol{L}_K\delta\boldsymbol c=0$, and also $\delta r=0$. The second equation in Eq.~\eqref{eq:matrix_pf1} therefore gives $\boldsymbol d=0$, and the first then gives $\delta\boldsymbol c=0$. Thus the homogeneous kernel is trivial, and Eq.~\eqref{eq:square_system_pf1} is nonsingular.
\end{proof}

Thm.~\ref{thm:uus} establishes the unique solvability of the fully discrete scheme proposed by us. We next show that this well-posedness is retained by the efficient two-solve formulation introduced in Section~\ref{sec:CARE-SAV}.
\begin{proposition}
\label{prop:well-posedness}
  Let
  $$
  \boldsymbol{A}_K=\boldsymbol{I}_K-\frac{\Delta t}{2}\boldsymbol{G}_K\boldsymbol{L}_K,
  \quad
  \boldsymbol{A}_K\boldsymbol q^{n+1/2}=\boldsymbol{G}_K\boldsymbol b^{n+1/2}.
  $$
  Under Assumption~\ref{as1}, $\boldsymbol{A}_K$ is nonsingular and
  \begin{equation}
    \label{eq:main_prop2}
    1-\frac{\Delta t}{4} (\boldsymbol b^{n+1/2})^{T}\boldsymbol q^{n+1/2}\ge1.
  \end{equation}
  Hence the scalar formula in Eq.~\eqref{eq:scalar_alpha} is well defined for every $\Delta t>0$.
\end{proposition}
\begin{proof}
  Proof see \cref{pf:prop:well-posedness}.
\end{proof}

Having established the well-posedness of both the CARE-SAV scheme and its efficient implementation, we next examine its stability properties, first through the discrete energy-dissipation law and subsequently from the perspective of numerical conditioning.

\subsection{Discrete Energy Stability and Conditioning}
We next examine the stability of the CARE-SAV scheme from two complementary perspectives. We first establish the discrete energy-dissipation law inherited from the underlying gradient-flow structure, and then characterize the numerical conditioning of the CARE representation and the resulting reduced solves.

\begin{theorem}[Energy Stability]
\label{thm:energy stability}
Define the modified energy by
  $$
    \widetilde E_K^n=\frac12(\boldsymbol c^n)^{T}\boldsymbol{L}_K\boldsymbol c^n + (r^n)^2-C_0.
  $$
  Then every solution of the fully discrete CARE-SAV scheme satisfies
  \begin{equation}
  \label{eq:main_pf2}
    \widetilde E_K^{n+1}-\widetilde E_K^n =\Delta t\,(\boldsymbol d^{n+1/2})^{T} \boldsymbol{G}_K\boldsymbol d^{n+1/2} \le 0.
  \end{equation}
 Consequently, the scheme is unconditionally energy stable with respect to the quadrature-based modified SAV energy.
\end{theorem}
\begin{proof}
  Taking the inner product of the phase equation with $\Delta t\,\boldsymbol d^{n+1/2}$ gives
  \begin{equation}
  \label{eq:pf2_1}
    (\boldsymbol c^{n+1}-\boldsymbol c^n)^{T}\boldsymbol d^{n+1/2}=\Delta t\,(\boldsymbol d^{n+1/2})^{T} \boldsymbol{G}_K\boldsymbol d^{n+1/2}.
  \end{equation}
  On the other hand, taking the inner product of the chemical-potential equation with $\boldsymbol c^{n+1}-\boldsymbol c^n$, and using the symmetry of $\boldsymbol{L}_K$, yields
  \begin{equation}
  \label{eq:pf2_2}
    \begin{aligned}
    (\boldsymbol c^{n+1}-\boldsymbol c^n)^{T}\boldsymbol d^{n+1/2}={}&\frac12\left[(\boldsymbol c^{n+1})^{T}\boldsymbol{L}_K\boldsymbol c^{n+1}-(\boldsymbol c^n)^{T}\boldsymbol{L}_K\boldsymbol c^n\right]\\
    &+r^{n+1/2}(\boldsymbol b^{n+1/2})^{T}(\boldsymbol c^{n+1}-\boldsymbol c^n).
    \end{aligned}
  \end{equation}
  Furthermore, multiplying the scalar equation by $2\Delta t\,r^{n+1/2}$ yields
  \begin{equation}
  \label{eq:pf2_3}
    (r^{n+1})^2-(r^n)^2=r^{n+1/2}(\boldsymbol b^{n+1/2})^{T}(\boldsymbol c^{n+1}-\boldsymbol c^n).
  \end{equation}
  Substituting Eq.~\eqref{eq:pf2_3} into Eq.~\eqref{eq:pf2_2}, we obtain Eq.~\eqref{eq:pf2_1}, which completes the proof.
\end{proof}

Thm.~\ref{thm:energy stability} shows that the CARE-SAV discretization preserves the energy dissipation structure of the underlying gradient flow without any restriction on the time-step size. We emphasize that the dissipated quantity is the modified SAV energy, which in general does not coincide exactly with the discrete physical energy under the linearly implicit time discretization. For conserved gradient flows, the reduced structure further leads to conservation of the discrete mass.

\begin{corollary}[Conditional Mass Conservation]
  Let $m_i=(1,\psi_i)_Q$ and $\mathcal M_Q^n=\boldsymbol m^{T}\boldsymbol c^n$.
  If $\boldsymbol m^{T}\boldsymbol{G}_K=0$, then for every time step, we have
  \begin{equation}
  \label{eq:main_cor1}
    \mathcal M_Q^{n+1}=\mathcal M_Q^n.
  \end{equation}
\end{corollary}
\begin{proof}
  Left-multiplication of the phase equation by $\boldsymbol m^{T}$ gives
  $$
  \frac{\mathcal M_Q^{n+1}-\mathcal M_Q^n}{\Delta t} =\boldsymbol m^{T}\boldsymbol{G}_K\boldsymbol d^{n+1/2}=0.
  $$
  Then Eq.~\eqref{eq:main_cor1} holds.
\end{proof}

A sufficient condition for the mass-conservation condition is that the constant function is exactly represented in the CARE space and the Cahn-Hilliard mobility operator is assembled in a structure-preserving manner.

The preceding results establish the structural stability of the CARE-SAV discretization through energy dissipation and, for conserved flows, mass conservation. We now complement these properties by examining the numerical conditioning of the CARE representation and the reduced linear systems arising in the efficient implementation.

\begin{proposition}[Conditioning of the CARE-SAV formulation]
\label{prop:Conditioning of the CARE-SAV formulation}
  Let $\kappa_2(B)=\sigma_{\max}(B)/\sigma_{\min}(B)$ for a nonsingular square matrix. In exact arithmetic, the CARE mass matrix satisfies
  \begin{equation}
  \label{eq:main1_prop2}
    \boldsymbol{M}_K^Q:=\Psi^{T} W\Psi=\boldsymbol{I}_K, \quad \kappa_2(\boldsymbol{M}_K^Q)=1.
  \end{equation}
  If a computed basis $\widetilde\Psi $ has orthogonality defect $\delta_Q:=\|\widetilde\Psi^{T} W\widetilde\Psi-\boldsymbol{I}_K\|_2<1$, then we have
  \begin{equation}
  \label{eq:main2_prop2}
    \kappa_2(\widetilde\Psi^{T} W\widetilde\Psi) \le \frac{1+\delta_Q}{1-\delta_Q}.
  \end{equation}
  Moreover, the scalar denominator satisfies
  \begin{equation}
  \label{eq:main3_prop2}
    D_n:=1-\frac{\Delta t}{4}(\boldsymbol b^{n+1/2})^{T}\boldsymbol q^{n+1/2}\ge1.
  \end{equation}
\end{proposition}
\begin{proof}
  Proof see \cref{sec:pf:prop:conditioning}.
\end{proof}

The preceding results establish both the structural stability and numerical robustness of the CARE-SAV discretization. We next quantify its approximation accuracy by deriving a fully discrete error estimate and the corresponding convergence result.

\subsection{Error Estimates and Convergence Analysis}
We first introduce several approximation quantities that characterize the spatial error induced by the random-feature dictionary and its subsequent CARE truncation. These quantities provide a unified measure of the approximation capability of the reduced CARE space and will be used in the fully discrete error estimate below.

With $W=W_N$ and $\Phi=\Phi_{M,N}$ denoting, respectively, the quadrature-weight matrix and feature-evaluation matrix introduced above, let
$$A_{M,N}:=W_N^{1/2}\Phi_{M,N},$$
and let $Q_K$ span the weighted sampled columns retained by CPQR. Define the computable tail quantity
$$\varepsilon_{M,N,K}^{\mathrm{QR}}:=\|(I_N-Q_K Q_K^{T})A_{M,N}\|_2.$$
For a target normed space $X$, define the sampling lift
$$\Lambda_M(X):=\sup_{0\ne v\in V_M^{\mathrm{RF}}}\frac{\|v\|_X}{\|v\|_{Q,N}},$$
with $\Lambda_M(X)=\infty$ when nodal sampling is not injective on $V_M^{\mathrm{RF}}$. Finally, define
$$\mathfrak a_{M,N,K}^{X}(f):=\inf_{\boldsymbol a\in\mathbb R^M}\left[\|f-\boldsymbol\varphi^{T}\boldsymbol a\|_X+\Lambda_M(X)\varepsilon_{M,N,K}^{\mathrm{QR}}\|\boldsymbol a\|_2\right].$$

The following lemma quantifies the additional approximation error introduced by the CARE truncation. In particular, it separates the approximation error of the original random-feature dictionary from the compression error induced by retaining only the leading CARE directions.
\begin{lemma}[CARE Truncation Estimate]
\label{lem:CARE truncation estimate}
  For every $v_M=\boldsymbol\varphi^{T}\boldsymbol a\in V_M^{\mathrm{RF}}$, we have
  \begin{equation}
  \label{eq:main1_lem1}
    \inf_{v_K\in V_K^{\mathrm{CARE}}}\|v_M-v_K\|_{Q,N}\le\varepsilon_{M,N,K}^{\mathrm{QR}}\|\boldsymbol a\|_2.
  \end{equation}
  Therefore, we obtain
  \begin{equation}
  \label{eq:main2_lem1}
    \inf_{v_K\in V_K^{\mathrm{CARE}}}\|f-v_K\|_X \le\mathfrak a_{M,N,K}^{X}(f).
  \end{equation}
\end{lemma}
\begin{proof}
  Proof see \cref{sec:pf:lem:CARE truncation estimate}.
\end{proof}

Lem.~\ref{lem:CARE truncation estimate} quantifies the approximation loss introduced by the CARE truncation, provided that the underlying random-feature dictionary is sufficiently expressive. We next establish that, under full-support random sampling, the Gaussian candidate spaces possess the required approximation capacity almost surely.

\begin{lemma}
\label{lem:almost-sure density at infinite width}
  Let $\Theta=\Omega\times[\sigma_{\min},\sigma_{\max}], 0<\sigma_{\min}<\sigma_{\max},$ and suppose the feature parameters $(\boldsymbol c_j,\sigma_j)$ are i.i.d. from a probability measure having full support on $\Theta$. If $\Omega$ is the closure of a bounded domain with nonempty interior and the ordinary Gaussian features in \eqref{eq:Euclidean-Gaussian-feature} are used, then
  \begin{equation}
    \overline{\bigcup_{M\ge1}V_M^{\mathrm{RF}}}^{\,C(\Omega)}=C(\Omega)
  \end{equation}
  almost surely holds. The same union is almost surely dense in $L^2(\Omega)$. If $\Omega$ is instead a periodic rectangular cell and the periodized Gaussian features in \eqref{eq:periodized-Gaussian-feature} are used, then
  \begin{equation}
    \overline{\bigcup_{M\ge1}V_M^{\mathrm{RF}}}^{\,C_{\mathrm{per}}(\overline\Omega)}
    =C_{\mathrm{per}}(\overline\Omega)
  \end{equation}
  almost surely, and the same union is almost surely dense in $L^2_{\mathrm{per}}(\Omega)$.
\end{lemma}
\begin{proof}
  Proof see \cref{sec:pf:lem:almost-sure density at infinite width}.
\end{proof}

To connect the CARE approximation with the fully discrete error analysis, let $P_K^Q$ be the quadrature-orthogonal projection onto $V_K^{\mathrm{CARE}}$, and define
$$
\phi_K^*(t):=P_K^Q\phi(t)=\boldsymbol\psi^{T}\boldsymbol c_*(t),
\quad
\mu_K^*(t):=P_K^Q\mu(t)=\boldsymbol\psi^{T}\boldsymbol d_*(t),
\quad
r_*(t):=r(t).
$$
Let $\mathfrak A_{M,N,K}(T)$ and $\mathfrak Q_{M,N,K}(T)$ denote the CARE approximation error and the quadrature-assembly consistency error respectively, as defined in \cref{sec:Projected Reference Solution and Consistency Estimates}. Under the regularity and projection-stability assumptions stated there, the residuals generated by inserting the projected exact trajectory into the reduced CARE-SAV system satisfy
\begin{equation}
  \begin{aligned}
    \sup_{0\le t\le T}\left(
    \|\boldsymbol\xi_{\mathcal G}(t)\|+\|\boldsymbol\xi_{\mathcal L,U}(t)\|+|\xi_r(t)|
    \right)
    &\le C\bigl[\mathfrak A_{M,N,K}(T)+\mathfrak Q_{M,N,K}(T)\bigr].
  \end{aligned}
\end{equation}
These consistency estimates allow the spatial approximation errors to be propagated through the fully discrete scheme, as quantified by the following stability lemma.

Throughout the following error analysis, $D_t v^n:=(v^{n+1}-v^n)/\Delta t$.

\begin{lemma}[Fixed-space Perturbation Stability]
\label{lem:Fixed-space Perturbation Stability}
  Let $(\boldsymbol c_*^n,r_*^n)$ be a reference sequence satisfying the CARE-SAV with derivative-scale residuals $\boldsymbol\rho_{c,n}$, $\boldsymbol\rho_{d,n}$ and $\rho_{r,n}$. Suppose that $\boldsymbol b_K$ is locally Lipschitz on the relevant bounded set, namely
  \begin{equation}
    \|\boldsymbol b_K(\boldsymbol x)-\boldsymbol b_K(\boldsymbol y)\|
    \le L_{b,K}\|\boldsymbol x-\boldsymbol y\|,
  \end{equation}
  and that $\|\boldsymbol b^{n+1/2}\|\le B_b$, $|r_*^{n+1/2}|\le B_r$ and $\|D_t\boldsymbol c_*^n\|\le B_t$. For sufficiently small $\Delta t$, we have
  \begin{equation}
    E_{n+1}\le(1+C_K\Delta t)E_n+C_K\Delta t E_{n-1} +C_K\Delta t\mathcal K_n,
  \end{equation}
  where
  $$E_n=\|\boldsymbol c_*^n-\boldsymbol c^n\|+|r_*^n-r^n|,\quad\mathcal K_n=\|\boldsymbol\rho_{c,n}\|+\|\boldsymbol\rho_{d,n}\|+|\rho_{r,n}|.$$
\end{lemma}
\begin{proof}
  Proof see \cref{sec:pf:lem:Fixed-space Perturbation Stability}
\end{proof}

The preceding approximation and stability results provide the ingredients needed for the global error analysis. Combining the CARE approximation estimates with the consistency bounds and the perturbation stability result yields the following fully discrete error estimate.

\begin{theorem}[Fully Discrete Error Estimate]
  \label{thm:fully-discrete-error-estimate}
  Assume that the candidate features satisfy the prescribed boundary conditions, so that $V_K^{\mathrm{CARE}}$ is contained in the corresponding admissible energy space; in particular, $V_K^{\mathrm{CARE}}\subset X_{\mathrm{per}}$ for periodic problems, and $P_K^Q$ is the projection onto this boundary-compatible space. Assume also that the exact solution is sufficiently smooth for the Crank-Nicolson and extrapolation residuals to be of derivative-scale order $O(\Delta t^2)$. Suppose that the projection and norm-equivalence bounds used to define $\mathfrak A_{M,N,K}(T)$ and $\mathfrak Q_{M,N,K}(T)$ hold, that the hypotheses of Lem.~\ref{lem:Fixed-space Perturbation Stability} are satisfied, and that $\Delta t$ is sufficiently small for its stability conclusion. If the initial error is consistent in the sense that
  \begin{equation}
    E_0\leq C\bigl[\mathfrak A_{M,N,K}(T)+\mathfrak Q_{M,N,K}(T)\bigr],
    \label{eq:initial-consistency}
  \end{equation}
  then, for exact linear solves,
  \begin{equation}
    \max_{t_n\leq T}\left(\|\phi(t_n)-\phi_K^n\|_{L^2}+|r(t_n)-r^n|\right)
    \leq C_{T,K}\left[\Delta t^2+\mathfrak A_{M,N,K}(T)+\mathfrak Q_{M,N,K}(T)\right],
    \label{eq:fully-discrete-error-estimate}
  \end{equation}
  where $C_{T,K}$ is independent of $\Delta t$.
\end{theorem}
\begin{proof}
  Let $(\boldsymbol c_*^n,\boldsymbol d_*^{n+1/2},r(t_n))$ be the projected exact trajectory. Taylor expansion of the Crank-Nicolson midpoint rule and of the extrapolation, together with the spatial consistency estimate, yields
  \begin{equation}
    \mathcal K_n\leq C\left[\Delta t^2+\mathfrak A_{M,N,K}(T)+\mathfrak Q_{M,N,K}(T)\right].
    \label{eq:reference-residual-bound}
  \end{equation}
  The scalar residual is retained as an independent error component, following the standard SAV error-analysis treatment~\cite{Chen2020SAVFEM}. The same estimate holds for the startup step; its derivative-scale extrapolation defect is $O(\Delta t)$ and is multiplied by one time step. Lem.~\ref{lem:Fixed-space Perturbation Stability}, with $E_{-1}=E_0$, and the discrete Gronwall inequality therefore give
  \begin{equation}
    \max_{t_n\leq T}\left(\|\boldsymbol c_*^n-\boldsymbol c^n\|+|r(t_n)-r^n|\right)
    \leq C_{T,K}\left[\Delta t^2+\mathfrak A_{M,N,K}(T)+\mathfrak Q_{M,N,K}(T)\right].
    \label{eq:coefficient-error-bound}
  \end{equation}
  Finally, decompose the phase error as
  \begin{equation}
    \phi(t_n)-\phi_K^n=\bigl[\phi(t_n)-P_K^Q\phi(t_n)\bigr]+\boldsymbol\psi^{T}\bigl(\boldsymbol c_*^n-\boldsymbol c^n\bigr).
  \end{equation}
  The first term is bounded by $C\mathfrak A_{M,N,K}(T)$, while the norm-equivalence bound on the fixed CARE space controls the second by the coefficient error in Eq.~\ref{eq:coefficient-error-bound}. Combining these estimates proves Eq.~\ref{eq:fully-discrete-error-estimate}.
\end{proof}

Thm.~\ref{thm:fully-discrete-error-estimate} directly implies convergence of the CARE-SAV approximation whenever the temporal, CARE approximation and quadrature errors vanish under joint refinement and the associated stability constants remain uniformly bounded. In particular, the scheme is second-order accurate in time for a fixed CARE space, while the spatial convergence is governed by the approximation quality of the CARE basis and the quadrature consistency.

Together, these results establish the well-posedness, structure preservation and convergence of the proposed CARE-SAV discretization.

\section{Numerical Experiments}
In this section, we investigate the numerical accuracy and energy stability of the proposed CARE-SAV framework through five representative gradient-flow models: Allen-Cahn equation, Cahn-Hilliard equation, the phase-field crystal model, the molecular beam epitaxy model and the grain growth model.

The common experimental configurations are given in \cref{sec:SM-experimental-setup}. Unless stated otherwise, the settings reported there are used throughout the numerical experiments.

\paragraph{Evaluation Metrics}
Against the reference solutions from \cref{sec:data-generator}, we report the relative $L^2$ and maximum $L^\infty$ errors over the recorded times:
\begin{equation}
  \mathrm{Rel.}L^2 = \left(\frac{\sum_{n\in\mathcal{S}}\|u_Q^n-u_{\mathrm{ref}}^n\|_{2,Q}^{2}}{\sum_{n\in\mathcal{S}}\|u_{\mathrm{ref}}^n\|_{2,Q}^{2}}\right)^{1/2}, \quad
  L^\infty=\max_{n\in\mathcal{S}}\|u_Q^n-u_{\mathrm{ref}}^n\|_{\infty}.
\end{equation}
Here, $u_Q^n$ and $u_{\mathrm{ref}}^n$ denote the CARE-SAV and reference solutions at $t_n$ respectively, and $\mathcal{S}$ is the set of recorded time indices.

\paragraph{Implementation Details}
Experiments use Python 3.10 with NumPy/SciPy in double precision on a 14-core Intel Xeon Gold 6430 CPU with 50 GB memory, using random seed 1234.

\subsection{Allen-Cahn Equation}
We first consider the one dimensional Allen-Cahn equation on $\Omega = (-1,1)$ and $t\in[0,1]$ as follows:
\begin{equation}
  u_t-\varepsilon^2u_{xx}+5u^3-5u=0,\quad u(-1,t)=u(1,t),\quad u(x,0)=x^2\cos(\pi x).
\end{equation}
We consider several values of $\varepsilon$ below; as $\varepsilon$ decreases, the diffuse interface becomes progressively thinner and the solution develops sharper spatial gradients, making the approximation increasingly challenging. Tab.~\ref{tab:ac_accuracy} demonstrates that CARE-SAV achieves high accuracy and computational efficiency. Fig.~\ref{fig:ac_accuracy} focuses on the most challenging case, $\varepsilon=10^{-5}$, and shows close agreement between the reference solution and the CARE-SAV approximation.

Besides, Fig.~\ref{fig:ac_energy}  shows the evolution of the physical energy and the shifted modified SAV energy. Both quantities decrease monotonically throughout the simulation, confirming the energy-dissipative behavior of the CARE-SAV scheme. Moreover, the two curves are nearly indistinguishable, indicating that the auxiliary variable remains highly consistent with the original nonlinear energy.

\begin{table}[htbp]
\centering
\caption{Accuracy and computational cost for the Allen-Cahn equation with different values of $\varepsilon$.}
\label{tab:ac_accuracy}
\small
\setlength{\tabcolsep}{10pt}
\renewcommand{\arraystretch}{1.15}
\begin{tabular}{cccc}
\toprule
$\varepsilon$
& Rel.\ $L^2$
& $L^\infty$
& CPU time (s)
\\
\midrule
$10^{-2}$
& 1.6949e-05
& 1.1294e-03
& 16.489
\\
$10^{-3}$
& 3.6207e-05
& 1.2466e-03
& 17.423
\\
$10^{-4}$
& 4.3234e-05
& 1.2679e-03
& 16.952
\\
$10^{-5}$
& 4.5126e-05
& 1.2694e-03
& 17.270
\\
\bottomrule
\end{tabular}
\end{table}

\begin{figure}[htbp]
  \centering
  \includegraphics[width=0.9\textwidth]{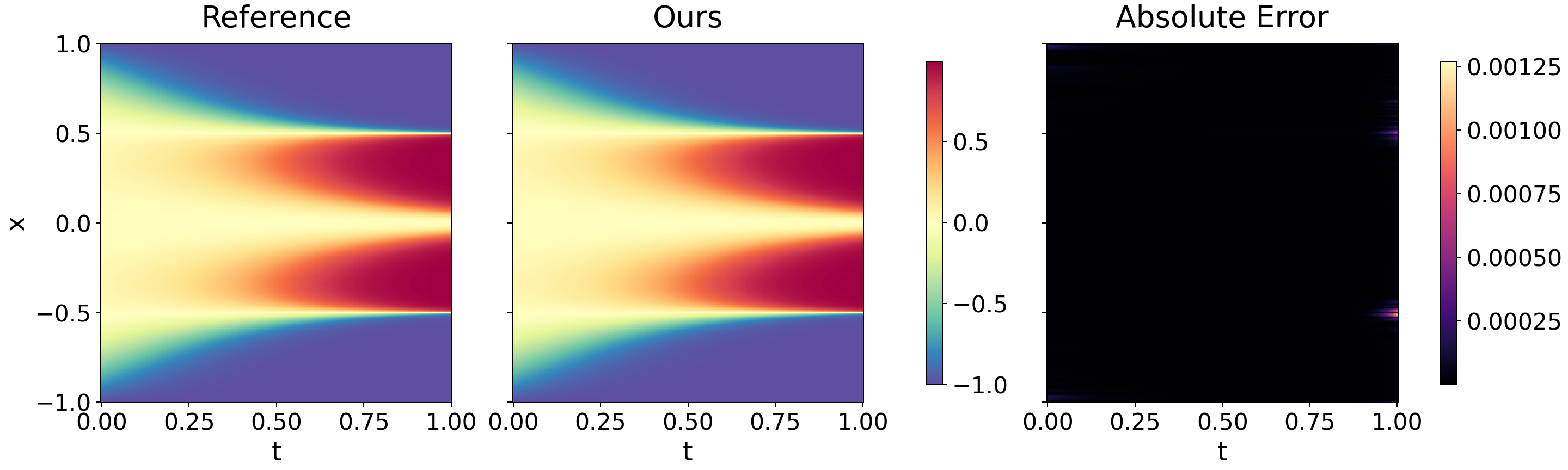}
  \caption{Reference solutions, CARE-SAV approximations, and pointwise absolute errors for the Allen-Cahn equation with $\varepsilon = 10^{-5}$.}
  \label{fig:ac_accuracy}
\end{figure}

\begin{figure}[htbp]
  \centering
  \includegraphics[width=0.9\textwidth]{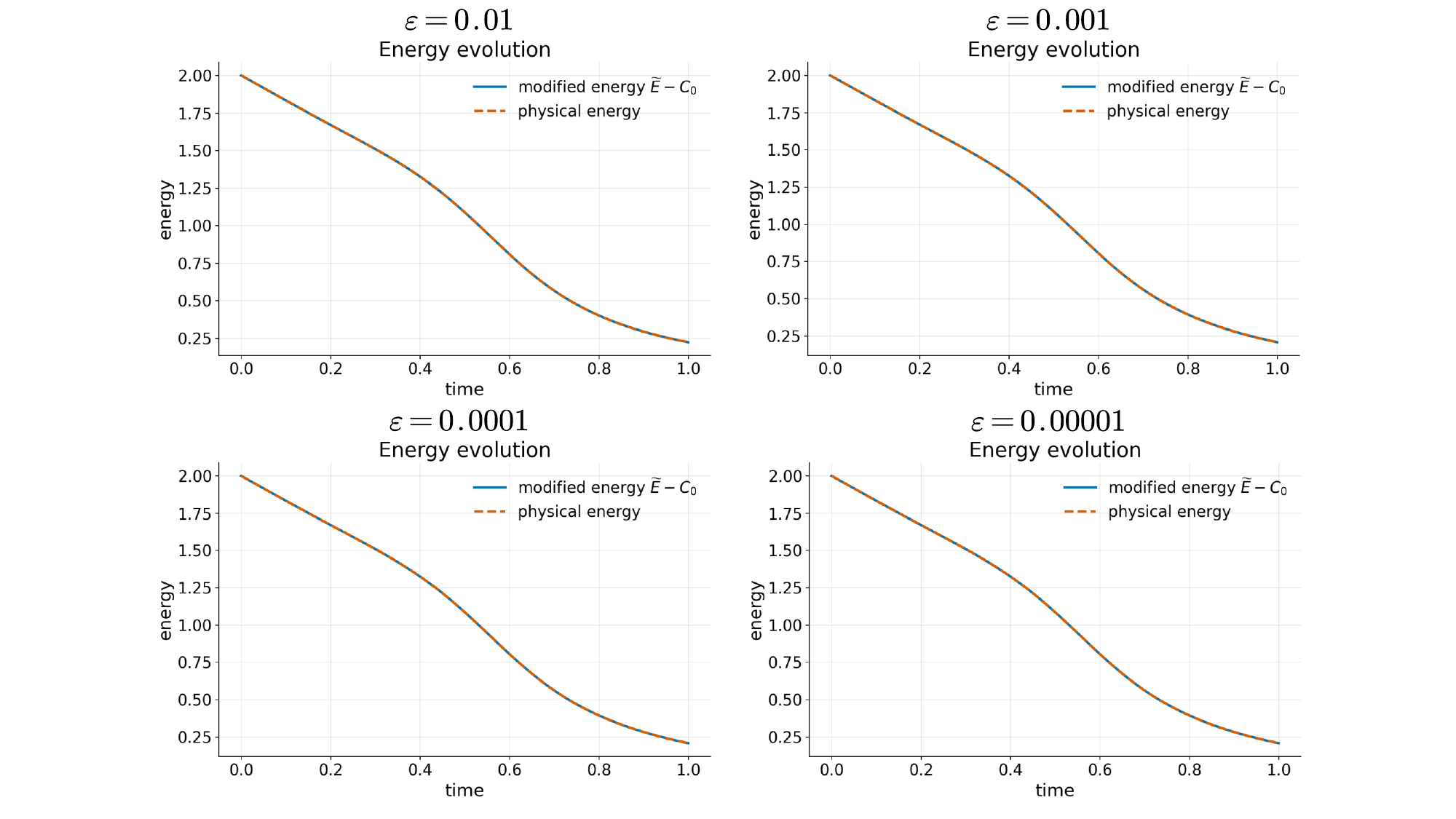}
  \caption{Energy evolution for the Allen-Cahn equation with different values of $\varepsilon$.}
  \label{fig:ac_energy}
\end{figure}

\subsection{Cahn-Hilliard Equation}
We next consider the two dimensional Cahn-Hilliard equation on the periodic domain $\Omega = [-1,1)^2$. Introducing the chemical potential $\mu$, the governing system is given by
\begin{equation}
u_t=\Delta\mu,\quad
\mu=-\varepsilon^2\Delta u+u^3-u, \quad t\in(0,T],
\end{equation}
where $T = 20 \varepsilon ^2$. Periodic boundary conditions are imposed in both spatial directions. The initial condition is chosen as a randomly perturbed field $u(x,y,0)=0.25+0.4 \mathrm{Rand}(x,y)$, where $\mathrm{Rand}(x,y)$ is a random number uniformly distributed in $[-1,1]$. This configuration produces spinodal decomposition followed by the gradual coarsening of phase-separated domains.

The corresponding free-energy functional is
\begin{equation}
E(u)=\int_{\Omega}\left[\frac{\varepsilon^2}{2}|\nabla u|^2+\frac{1}{4}(u^2-1)^2\right]\mathrm{d}\boldsymbol{x}.
\end{equation}
The Cahn-Hilliard dynamics conserve the total mass while dissipating the free energy:
\begin{equation}
\frac{\mathrm{d}}{\mathrm{d}t}\int_{\Omega}u\,\mathrm{d}\boldsymbol{x}=0,
\quad
\frac{\mathrm{d}E}{\mathrm{d}t}=-\int_{\Omega}|\nabla\mu|^2\,\mathrm{d}\boldsymbol{x} \le 0.
\end{equation}

Tab.~\ref{tab:ch_accuracy} demonstrates that CARE-SAV achieves high accuracy and computational efficiency. Fig.~\ref{fig:ch_accuracy} focuses on the most challenging case, $\varepsilon=0.1$, and shows close agreement between the reference solution and the CARE-SAV approximation at $t=T$.

Fig.~\ref{fig:ch_energy} presents the energy behavior of CARE-SAV for the Cahn-Hilliard equation. Both the physical energy and the modified SAV energy decrease monotonically throughout the simulation. The two energy curves are visually indistinguishable, demonstrating excellent consistency between the auxiliary variable and the original nonlinear energy.

\begin{table}[htbp]
\centering
\caption{Accuracy and computational cost for the Cahn-Hilliard equation with different values of $\varepsilon$.}
\label{tab:ch_accuracy}
\small
\setlength{\tabcolsep}{10pt}
\renewcommand{\arraystretch}{1.15}
\begin{tabular}{cccc}
\toprule
$\varepsilon$
& Rel.\ $L^2$
& $L^\infty$
& CPU time (s)
\\
\midrule
0.1
& 3.7604e-07
& 1.0079e-06
& 85.892
\\
0.2
& 1.1951e-07
& 3.2017e-07
& 83.862
\\
0.5
& 1.5108e-07
& 1.6376e-07
& 98.755
\\

\bottomrule
\end{tabular}
\end{table}

\begin{figure}[htbp]
  \centering
  \includegraphics[width=0.9\textwidth]{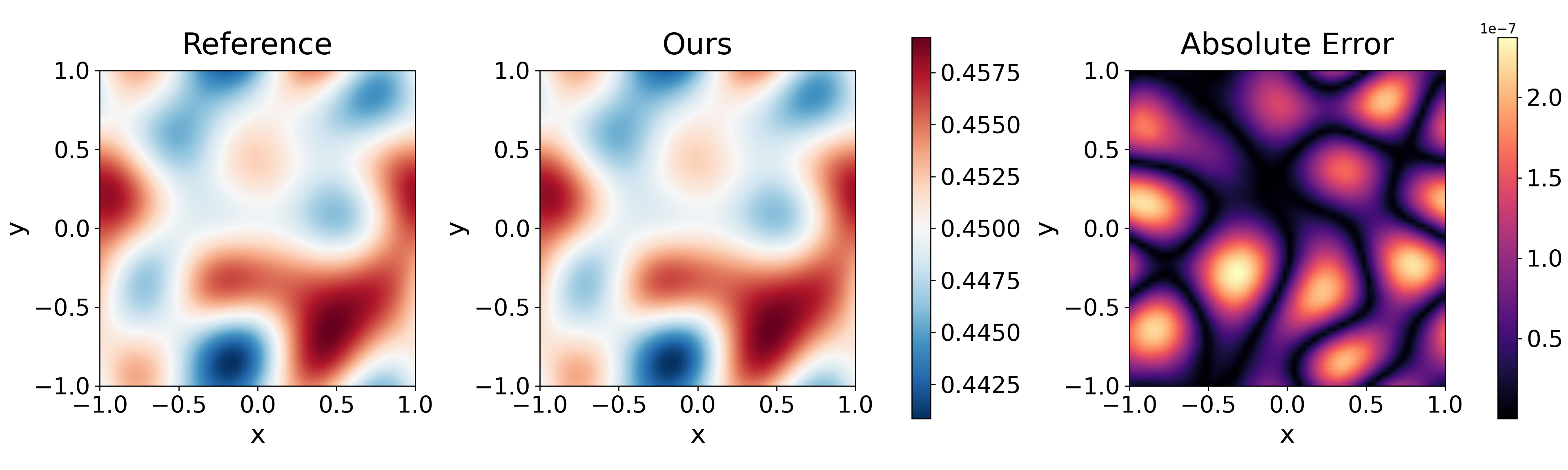}
  \caption{Reference solutions, CARE-SAV approximations, and pointwise absolute errors for the Cahn-Hilliard equation with $\varepsilon = 0.1$ when $t=T$.}
  \label{fig:ch_accuracy}
\end{figure}

\begin{figure}[htbp]
  \centering
  \includegraphics[width=1\textwidth]{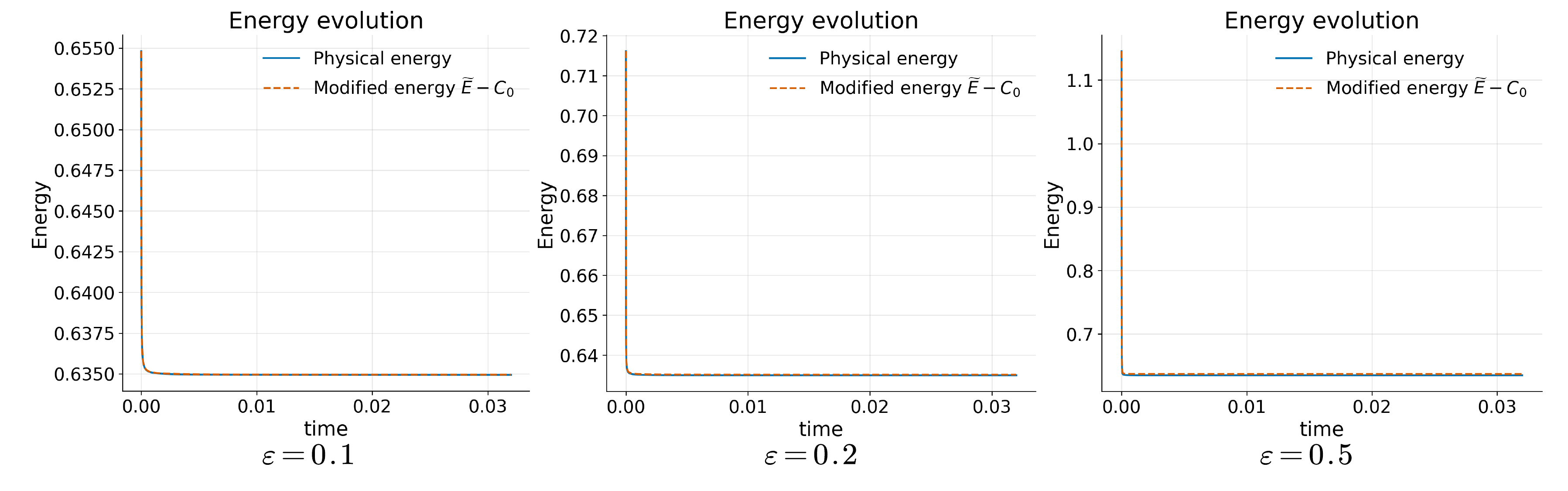}
  \caption{Energy evolution for the Cahn-Hilliard equation with different values of $\varepsilon$.}
  \label{fig:ch_energy}
\end{figure}

\subsection{Phase-Field Crystal Model}

We next consider the dimensionless two-dimensional phase-field crystal (PFC) model on the periodic domain $\Omega = [0,32]^2$ and $T=10$. The governing equations are
\begin{equation}
  u_t=\Delta\mu,
  \quad
  \mu=\left[r+(\Delta+1)^2\right]u+u^3,
  \quad
  (x,y)\in\Omega,\quad t\in(0,T],
\end{equation}
with the initial condition
\begin{equation}
  u(x,y,0)=0.5\sin\left(\frac{2\pi x}{32}\right)\sin\left(\frac{2\pi y}{32}\right).
\end{equation}
This conserved sixth-order gradient-flow problem provides a more challenging test of the ability of CARE-SAV to represent high-order spatial operators while preserving the energy-dissipative structure.The corresponding free-energy functional is
\begin{equation}
  E(u)=\int_{\Omega}\left[\frac{1}{2}u\left(r+(\Delta+1)^2\right)u+\frac{1}{4}u^4\right]\mathrm{d}\boldsymbol{x}.
\end{equation}

The parameter $r$ is a dimensionless control parameter commonly associated with the reduced temperature or degree of undercooling in the PFC model. For $r<0$, spatial modes near the preferred wave number become linearly unstable, and decreasing $r$ further strengthens this instability and promotes more pronounced crystalline structures. Therefore, more negative values of $r$ generally lead to stronger spatial variations and a more challenging approximation problem.

Tab.~\ref{tab:pfc_accuracy} demonstrates that CARE-SAV achieves high accuracy and computational efficiency. Fig.~\ref{fig:pfc_accuracy} focuses on the most challenging case, $r=-0.5$, and shows that the method accurately captures the corresponding spatial structures at $t=T$.

\begin{table}[htbp]
\centering
\caption{Accuracy and computational cost for the PFC with different values of $r$.}
\label{tab:pfc_accuracy}
\small
\setlength{\tabcolsep}{10pt}
\renewcommand{\arraystretch}{1.15}
\begin{tabular}{cccc}
\toprule
$r$
& Rel.\ $L^2$
& $L^\infty$
& CPU time (s)
\\
\midrule
-0.1
& 4.1516e-05
& 4.9831e-05
& 135.79
\\
-0.3
& 1.3086e-04
& 2.3999e-04
& 141.61
\\
-0.5
& 5.8187e-04
& 2.1641e-03
& 133.85
\\

\bottomrule
\end{tabular}
\end{table}

\begin{figure}[htbp]
  \centering
  \includegraphics[width=0.9\textwidth]{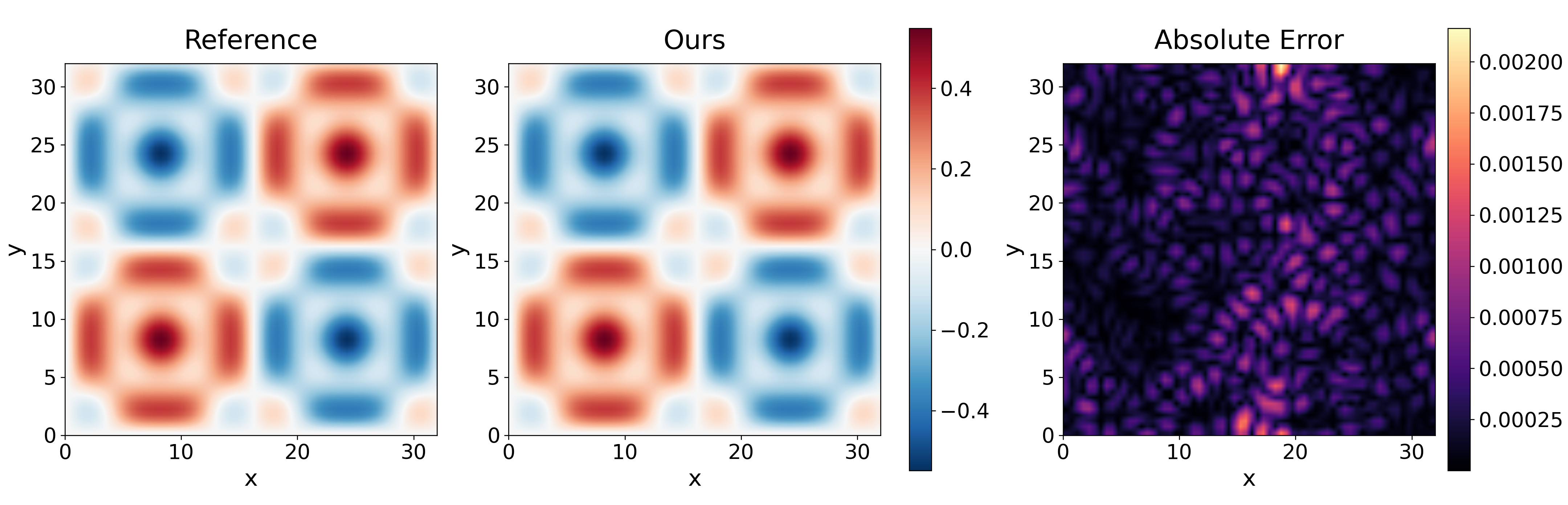}
  \caption{Reference solutions, CARE-SAV approximations and pointwise absolute errors for the PFC with $r = -0.5$ when $t=T$.}
  \label{fig:pfc_accuracy}
\end{figure}

Moreover, Fig.~\ref{fig:pfc_en} shows the evolution of the physical energy and the modified SAV energy for different values of $r$. Both energies decrease monotonically over time, confirming the energy-dissipative nature of the CARE-SAV scheme. The two curves are nearly indistinguishable, indicating that the auxiliary variable closely tracks the original nonlinear energy throughout the simulation.

\begin{figure}[htbp]
  \centering
  \includegraphics[width=1\textwidth]{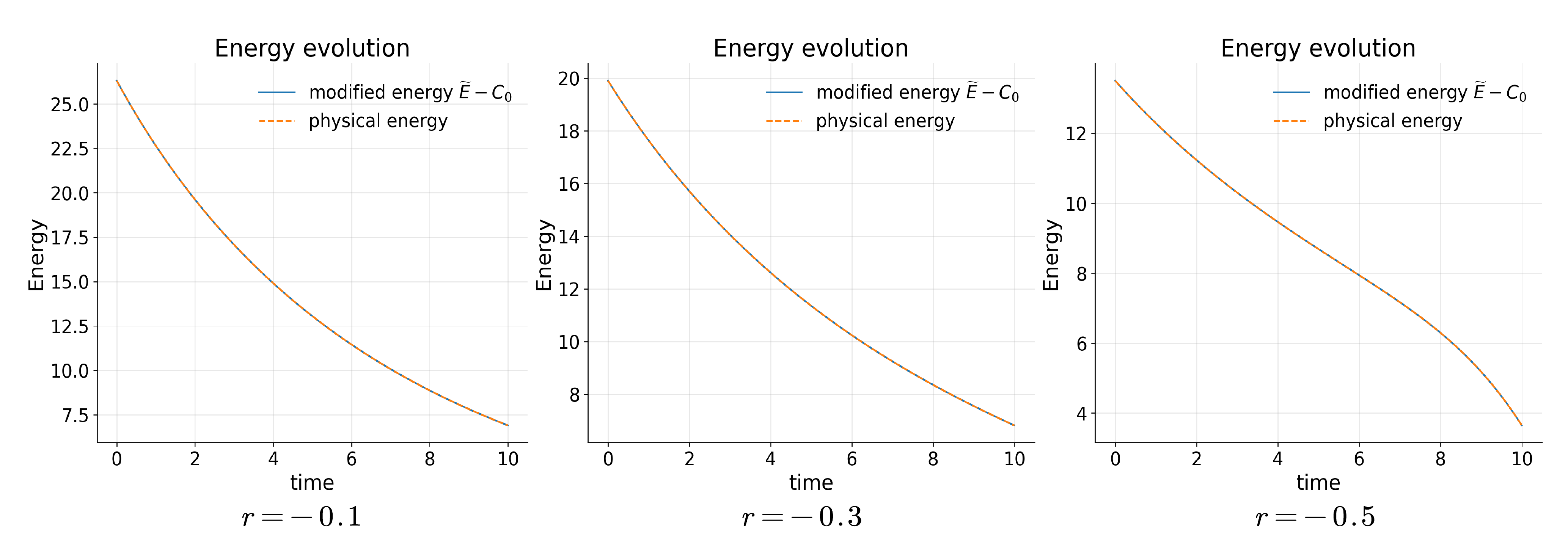}
  \caption{Energy evolution for the PFC with different values of $r$.}
  \label{fig:pfc_en}
\end{figure}

\subsection{Molecular Beam Epitaxy Model}
We next consider the two-dimensional molecular beam epitaxy (MBE) model
with slope selection on the periodic domain
$\Omega=[0,2\pi)^2$. To evaluate the numerical accuracy of CARE-SAV,
we introduce a forcing term and employ the method of manufactured
solutions. The governing equation is
\begin{equation}
  u_t=-\varepsilon^2\Delta^2u-\Delta u+\nabla\cdot\left(|\nabla u|^2\nabla u\right)+f(x,y,t), \quad(x,y,t)\in\Omega\times(0,T],
\end{equation}
with initial condition $u(x,y,0)=A\sin x\sin y$. Periodic boundary conditions are imposed in both spatial directions. The source term is chosen such that the exact solution is
\begin{equation}
  u(x,y,t)=A\cos t\sin x\sin y.
\end{equation}

Specifically,
\begin{equation}
\begin{aligned}
f(x,y,t)={}&-A\bigl[\sin t+(2-4\varepsilon^2)\cos t\bigr]\sin x\sin y \\
&+\frac{A^3\cos^3 t}{4}\Bigl[5\sin x\sin y+\sin x\sin(3y) \\
&\qquad\quad{}+\sin(3x)\sin y-3\sin(3x)\sin(3y)\Bigr].
\end{aligned}
\end{equation}
In the numerical experiment, we set
$A=0.1, \varepsilon^2=0.1, T=1$.
This smooth manufactured solution provides a direct test of the ability of CARE-SAV to approximate the fourth-order linear operator and the nonlinear gradient term appearing in the MBE model.

Tab.~\ref{tab:mbe_accuracy} demonstrates that CARE-SAV achieves high accuracy and computational efficiency. Fig.~\ref{fig:mbe_acc} compares the reference solution and CARE-SAV approximation and displays the energy evolution, confirming accurate resolution of the MBE dynamics.

\begin{table}[htbp]
\centering
\caption{Accuracy and computational cost for the MBE.}
\label{tab:mbe_accuracy}
\small
\setlength{\tabcolsep}{10pt}
\renewcommand{\arraystretch}{1.15}
\begin{tabular}{ccc}
\toprule

Rel.\ $L^2$
& $L^\infty$
& CPU time (s)
\\
\midrule

6.0786e-08
& 1.0722e-08
& 106.26

\\

\bottomrule
\end{tabular}
\end{table}

\begin{figure}[htbp]
  \centering
  \includegraphics[width=1\textwidth]{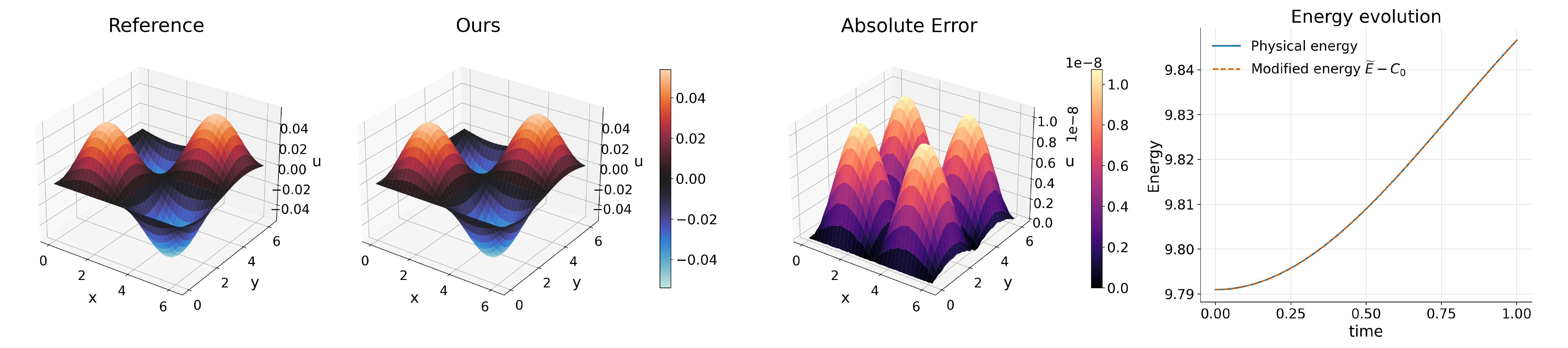}
  \caption{Reference solutions, CARE-SAV approximations and pointwise absolute errors for the MBE, 3D visualization when $t=T$ and the energy evolution.}
  \label{fig:mbe_acc}
\end{figure}

\begin{remark}
  The forcing term $f(x,y,t)$ is introduced solely to construct the prescribed exact solution. Consequently, the resulting problem is no longer an unforced gradient flow, and its physical energy is not necessarily monotone. Therefore, this example is used primarily to assess numerical accuracy rather than energy dissipation.
\end{remark}

\subsection{Grain Growth Model}
We finally consider the multi-order-parameter grain growth model of Fan and Chen \cite{FAN1997611} on the periodic domain $\Omega=[0,2\pi)^2$. With $m$ nonconserved order parameters $\boldsymbol{\eta}= (\eta_1,\cdots,\eta_m)$, the model is the $L^2$ gradient flow
\begin{equation}
  \partial_t\eta_i=-L\frac{\delta E}{\delta\eta_i},\quad i=1,\cdots,m,
\end{equation}
associated with
\begin{equation}
  E(\boldsymbol{\eta})=\int_{\Omega}\left[\frac{\kappa}{2}\sum_{i=1}^{m}|\nabla\eta_i|^2 + f_{\mathrm{bulk}}(\boldsymbol{\eta})\right]\mathrm{d}\boldsymbol{x},
\end{equation}
where
\begin{equation}
f_{\mathrm{bulk}}(\boldsymbol{\eta})=-\frac{a}{2}\sum_{i=1}^{m}\eta_i^2 + \frac{b}{4}\sum_{i=1}^{m}\eta_i^4 + \gamma \sum_{1\leq i<j\leq m} \eta_i^2\eta_j^2,
\end{equation}
and periodic boundary conditions are imposed in both spatial directions. The model satisfies
\begin{equation}
\frac{\mathrm{d}E}{\mathrm{d}t} =-\frac{1}{L} \sum_{i=1}^{m}\left\|\partial_t\eta_i \right\|_{L^2(\Omega)}^2 \leq 0.
\end{equation}

Fig.~\ref{fig:grain_eta_max} shows the evolution of $\eta_{\max}(\boldsymbol{x},t)=\max_{1\leq i\leq m}\eta_i(\boldsymbol{x},t)$. Large values correspond to grain interiors, while low-value regions identify the grain boundaries. As time evolves, the boundaries migrate and become smoother, with small grains shrinking and larger grains expanding. The results demonstrate that CARE-SAV captures the characteristic grain-coarsening dynamics driven by free-energy dissipation.

\begin{figure}[htbp]
  \centering
  \includegraphics[width=0.8\textwidth]{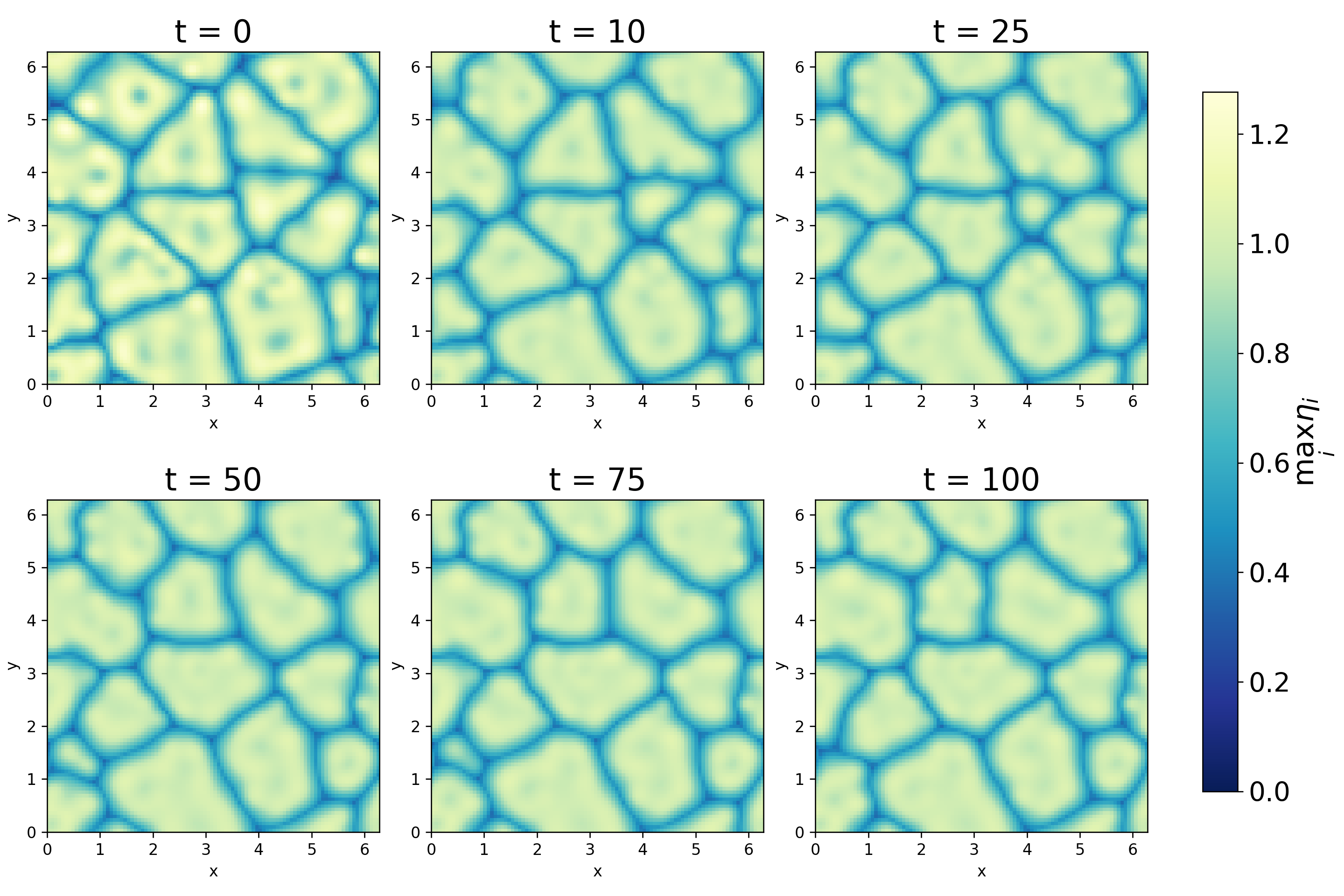}
  \caption{Evolution of the maximum order parameter $\eta_{\max}(\boldsymbol{x},t)=\max_{1\leq i\leq m}\eta_i(\boldsymbol{x},t)$ for the grain-growth model at different times.}
  \label{fig:grain_eta_max}
\end{figure}

Fig.~\ref{fig:grain_statistics} summarizes the grain-growth dynamics through three complementary diagnostics. The active grain count decreases stepwise from 16 to 13, indicating the gradual disappearance of unfavorable grains during the coarsening process. Meanwhile, both the physical energy and the modified SAV energy decrease monotonically and remain visually indistinguishable throughout the simulation, confirming the energy-dissipative behavior and the consistency of the auxiliary variable. The normalized grain-area distribution, expressed in terms of $A/\bar{A}$, also evolves over time as the grain structure reorganizes. These results are consistent with the progressive elimination of small grains and the growth of the remaining grains.

\begin{figure}[htbp]
  \centering
  \includegraphics[width=1\textwidth]{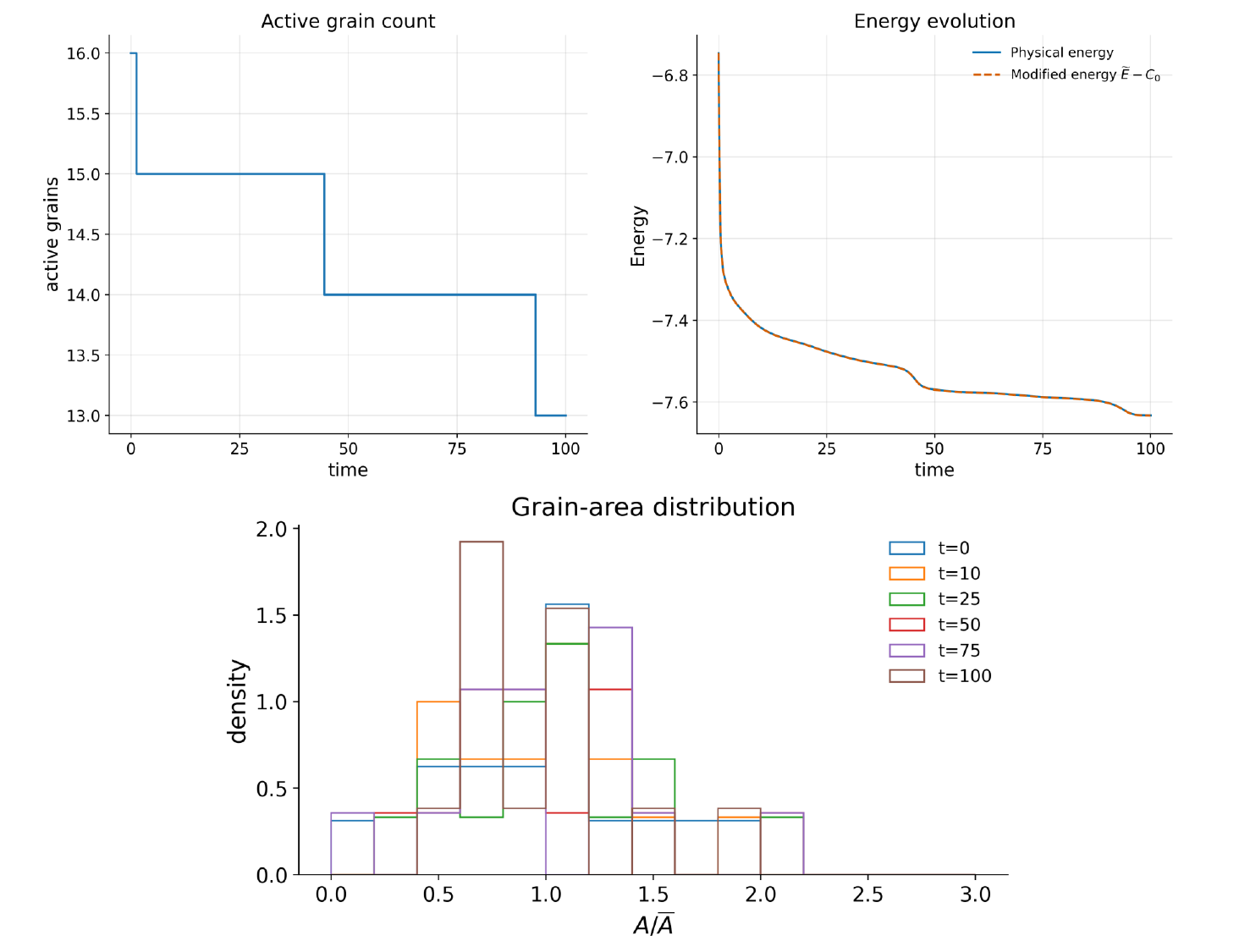}
  \caption{Evolution of the active grain count, the physical and modified SAV energies and the normalized grain-area distribution for the grain growth model.}
  \label{fig:grain_statistics}
\end{figure}

\subsection{Parameter Study}
The influence of representation parameters on CARE-SAV is examined through the choice of feature family, the number of candidate features, and the rank-revealing QR tolerance. To keep the study focused, the analysis is restricted to the Allen-Cahn and Cahn-Hilliard equations, which serve as representative nonconserved and conserved gradient flows, respectively.

\paragraph{Feature-Family Comparison}
To assess the influence of the feature family, we compare the Gaussian RBF representation adopted in CARE-SAV with Tanh-ELM \cite{CALABRO2021}, Fourier-ELM \cite{ren2025GFF} and IMQ-RBF \cite{NASSAJIANMOJARRAD2023151}. For a fair comparison, all numerical settings other than the feature family are kept identical. The relative $L^2$ error is used as the sole evaluation metric, thereby isolating the effect of different feature constructions on approximation accuracy. Tab.~\ref{tab:feature_family_comparison} shows that Ours achieves the lowest relative $L^2$ error in all tested cases.

\begin{table}[htbp]
\centering
\caption{Comparison of relative $L^2$ errors obtained using different
feature families under identical numerical settings, Promotion denotes the percentage reduction in relative $L^2$ error achieved by Ours compared with the second best method.}
\label{tab:feature_family_comparison}
\small
\setlength{\tabcolsep}{6pt}
\renewcommand{\arraystretch}{1.15}
\resizebox{\linewidth}{!}{%
\begin{tabular}{c|c|ccc|cc}
\toprule
Equation
& $\varepsilon$
& Tanh-ELM
& Fourier-ELM
& IMQ-RBF
& \textbf{Ours}
& Promotion
\\
\midrule

\multirow{4}{*}{Allen-Cahn}
& $10^{-2}$
& 3.48e-02
& 8.18e-02
& \underline{2.51e-04}
& \textbf{1.69e-05}
& 93.27\%
\\
& $10^{-3}$
& 4.51e-02
& 9.12e-02
& \underline{1.39e-04}
& \textbf{3.62e-05}
& 73.96\%
\\
& $10^{-4}$
& 4.53e-02
& 9.11e-02
& \underline{1.51e-04}
& \textbf{4.32e-05}
& 71.39\%
\\
& $10^{-5}$
& 4.53e-02
& 9.23e-02
& \underline{1.56e-04}
& \textbf{4.51e-05}
& 71.09\%
\\

\midrule

\multirow{3}{*}{Cahn-Hilliard}
& $0.1$
& 4.86e-03
& 1.49e-02
& \underline{4.86e-03}
& \textbf{1.51e-07}
& 99.99\%
\\
& $0.2$
& 2.83e-03
& 2.57e-02
& \underline{2.83e-03}
& \textbf{1.19e-07}
& 99.99\%
\\
& $0.5$
& 2.03e-03
& 4.31e-02
& \underline{2.03e-03}
& \textbf{3.76e-07}
& 99.98\%
\\

\bottomrule
\end{tabular}
}
\end{table}

\paragraph{Parameter Sensitivity}
The sensitivity studies in \cref{sec:SM-parameter-sensitivity} show that increasing the number of candidate features improves accuracy with diminishing returns once the CARE space is sufficiently rich. Moreover, the accuracy remains stable over the tested range of rank-revealing QR tolerances, indicating that CARE-SAV can remove numerically redundant feature directions without substantial loss of accuracy.

\section{Discussions and Conclusions}
Overall, the results presented in this work suggest that flexible feature representations need not be separated from classical structure-preserving discretization principles. Once the generated feature space is treated as a genuine Galerkin space, the dissipative structure of gradient flows can be retained within a nontraditional spatial representation rather than being tied to a predetermined mesh or basis. This perspective indicates that representation flexibility and rigorous numerical structure can coexist within the same discretization framework, providing a possible route toward more adaptable structure-preserving solvers for dissipative PDEs.

On the other hand, several questions remain open beyond the present study. In particular, the construction of the CARE space may be further adapted to the evolving solution, allowing the spatial representation to respond dynamically to changes in solution structure. Extending the same variational perspective to broader classes of dissipative PDEs also provides a natural direction for future investigation.

Moreover, a further direction is to investigate whether the feature-space perspective developed here can support reduced or data-driven descriptions of dissipative PDE dynamics. In particular, understanding how structural information can be retained in finite-dimensional dynamical representations may provide a useful basis for future extensions beyond direct time integration.

\appendix

\section{Proof of Proposition~\ref{prop:well-posedness}}
\label{pf:prop:well-posedness}
\begin{proof}
  Suppose $\boldsymbol{A}_K\boldsymbol x=0$. Then
  $$
    \boldsymbol x=\frac{\Delta t}{2}\boldsymbol{G}_K\boldsymbol{L}_K\boldsymbol x.
  $$
  Taking the inner product with $\boldsymbol{L}_K\boldsymbol x$ gives
  $$
    \boldsymbol x^{T}\boldsymbol{L}_K\boldsymbol x = \frac{\Delta t}{2}(\boldsymbol{L}_K\boldsymbol x)^{T} \boldsymbol{G}_K(\boldsymbol{L}_K\boldsymbol x)\le0.
  $$
  Since the left-hand side is nonnegative, $\boldsymbol{L}_K\boldsymbol x=0$, and the defining equation then yields $\boldsymbol x=0$. Thus $\boldsymbol{A}_K$ is nonsingular. Set
  $$
    \boldsymbol z=\boldsymbol b^{n+1/2}+\frac{\Delta t}{2}\boldsymbol{L}_K\boldsymbol q^{n+1/2}.
  $$
  The equation for $\boldsymbol q^{n+1/2}$ is equivalent to $\boldsymbol q^{n+1/2}=\boldsymbol{G}_K\boldsymbol z$. Therefore we have
  \begin{equation}
  \label{eq:SM1_prop1}
    \begin{aligned}
    (\boldsymbol b^{n+1/2})^{T}\boldsymbol q^{n+1/2} &=\boldsymbol z^{T}\boldsymbol q^{n+1/2} -\frac{\Delta t}{2}(\boldsymbol q^{n+1/2})^{T} \boldsymbol{L}_K\boldsymbol q^{n+1/2}\\
    &=\boldsymbol z^{T}\boldsymbol{G}_K\boldsymbol z -\frac{\Delta t}{2}(\boldsymbol q^{n+1/2})^{T} \boldsymbol{L}_K \boldsymbol q^{n+1/2} \le 0.
    \end{aligned}
  \end{equation}
  Thus Prop.~\ref{prop:well-posedness} is proved.
\end{proof}

\section{Proof of Proposition~\ref{prop:Conditioning of the CARE-SAV formulation}}
\label{sec:pf:prop:conditioning}
\begin{proof}
  The first identity follows directly from the discrete orthonormality of the CARE basis, $\Psi^{T} W\Psi=\boldsymbol{I}_K$, and hence
  \begin{equation}
    \kappa_2(\boldsymbol{M}_K^Q)=1.
  \end{equation}
  For a numerically computed basis $\widetilde\Psi$ the assumption $\delta_Q:=\|\widetilde\Psi^{T} W\widetilde\Psi-\boldsymbol{I}_K\|_2<1$ implies that all eigenvalues of $\widetilde\Psi^{T} W\widetilde\Psi$ lie in the interval $[1-\delta_Q,1+\delta_Q]$. Therefore,
  \begin{equation}
    \kappa_2(\widetilde\Psi^{T} W\widetilde\Psi) \le \frac{1+\delta_Q}{1-\delta_Q}.
  \end{equation}
  Finally, recall that Eq.~\eqref{eq:SM1_prop1} holds, where the last inequality follows from the dissipativity of $\boldsymbol{G}_K$ and the positive semidefiniteness of $\boldsymbol{L}_K$. Consequently,
  \begin{equation}
    D_n = 1-\frac{\Delta t}{4} (\boldsymbol b^{n+1/2})^{T} \boldsymbol q^{n+1/2} \ge 1.
  \end{equation}
\end{proof}

\section{Proof of Lemma~\ref{lem:CARE truncation estimate}}
\label{sec:pf:lem:CARE truncation estimate}
\begin{proof}
  The weighted nodal vectors of selected CARE functions span $\operatorname{range}(Q_K)$. Hence the best nodal approximation to $A_{M,N}\boldsymbol a$ from this range is its orthogonal projection $Q_KQ_K^{T} A_{M,N}\boldsymbol a$, and
  $$
  \inf_{v_K\in V_K^{\mathrm{CARE}}}\|v_M-v_K\|_{Q,N}=\|(I_N-Q_KQ_K^{T})A_{M,N}\boldsymbol a\|_2.
  $$
  The spectral-norm estimate gives Eq.~\eqref{eq:main1_lem1}. Since $v_M-v_K\in V_M^{\mathrm{RF}}$, the definition of $\Lambda_M(X)$, the triangle inequality, and minimization over $\boldsymbol a$ give Eq.~\eqref{eq:main2_lem1}.
\end{proof}

\section{Proof of Lemma~\ref{lem:almost-sure density at infinite width}}
\label{sec:pf:lem:almost-sure density at infinite width}
\begin{proof}
  Because $\Theta$ is a separable compact metric space, it has a countable topological base $\{O_j\}_{j\ge1}$. Full support gives $p_j=\mathbb P\{\theta_1\in O_j\}>0$, and
  $$\mathbb P\{\theta_1,\cdots,\theta_M\notin O_j\}=(1-p_j)^M \longrightarrow 0.$$
  Taking the countable intersection shows that the sampled parameter sequence is almost surely dense in $\Theta$.

  Fix any $\sigma_0\in(\sigma_{\min},\sigma_{\max})$. In the nonperiodic setting, the universal approximation result for fixed-width Gaussian radial basis functions \cite{PARKSANDBERG1991246} shows that finite linear combinations of the ordinary Gaussian RBFs are dense in $C(\Omega)$. The parameter-to-feature map is continuous from $\Theta$ to $C(\Omega)$. Almost-sure parameter density therefore allows every deterministic Gaussian section, and hence every finite combination of such sections, to be approximated by sampled features. This proves $\overline{\bigcup_{M\ge1}V_M^{\mathrm{RF}}}^{\,C(\Omega)}=C(\Omega)$ almost surely. Density in $L^2(\Omega)$ follows from the density of continuous functions in $L^2(\Omega)$.

  For a periodic rectangular cell, identify opposite faces and write the domain as the torus $\mathbb T_{\boldsymbol L}^d$. The periodized fixed-width Gaussian
  $$
  g_{\sigma_0}^{\mathrm{per}}(\boldsymbol x)=\sum_{\boldsymbol k\in\mathbb Z^d}\exp\left(-\frac{\|\boldsymbol x+\boldsymbol k\odot\boldsymbol L\|^2}{2\sigma_0^2}\right)
  $$
  has a nonzero Fourier coefficient at every frequency. Consequently, the closed linear span of its translates contains all trigonometric polynomials and is dense in $C(\mathbb T_{\boldsymbol L}^d)=C_{\mathrm{per}}(\overline\Omega)$. The map $(\boldsymbol c,\sigma)\mapsto g_\sigma^{\mathrm{per}}(\cdot-\boldsymbol c)$ is continuous from $\Theta$ to $C_{\mathrm{per}}(\overline\Omega)$. The same almost-sure parameter-density argument therefore gives
  $$
  \overline{\bigcup_{M\ge1}V_M^{\mathrm{RF}}}^{\,C_{\mathrm{per}}(\overline\Omega)}=C_{\mathrm{per}}(\overline\Omega).
  $$
  Finally, density in $L^2_{\mathrm{per}}(\Omega)$ follows from the density of periodic continuous functions in that space.
\end{proof}

\section{Projected Reference Solution and Consistency Estimates}
\label{sec:Projected Reference Solution and Consistency Estimates}
By using the physical SAV variable as the scalar reference and introducing the projected-energy quantity only for comparison, set
$$
r_*(t):=r(t)=\sqrt{E_1(\phi(t))+C_0},
\quad
\widehat r_K(t):=\sqrt{E_{1,Q}(\phi_K^*(t))+C_0}.
$$
Since the CARE space is fixed in time, $\partial_tP_K^Q\phi=P_K^Q\partial_t\phi$. The physical reference $r(t)$ does not satisfy the reduced scalar equation exactly. Define its nonzero spatial consistency defect by
\begin{equation}\label{eq:SM nonzero spatial consistency}
\begin{aligned}
\xi_r(t)&:=\dot r(t)-\frac12\boldsymbol b_K(\boldsymbol c_*(t))^{T}\dot{\boldsymbol c}_*(t)\\
&=\frac12\left[\frac{(\mathcal U(\phi(t)),\partial_t\phi(t))}{r(t)}-\frac{(\mathcal U(\phi_K^*(t)),\partial_t\phi_K^*(t))_Q}{\widehat r_K(t)}\right].
\end{aligned}
\end{equation}
Thus no zero-SAV-defect assumption is imposed. Treating $r(t_n)-r^n$ as an independent error component and retaining its scalar residual follows the standard SAV error-analysis framework~\cite{SHENXU20182895,Chen2020SAVFEM}. Let $J_Kv=((v,\psi_i)_Q)_{i=1}^K$. Under ideal strong-form quadrature assembly, the two vector defects are
\begin{equation}
  \begin{aligned}
    \boldsymbol\xi_{\mathcal G}(t)&=\dot{\boldsymbol c}_*(t)-\boldsymbol{G}_K\boldsymbol d_*(t)=J_K\mathcal G(I-P_K^Q)\mu(t),\\
    \boldsymbol\xi_{\mathcal L,U}(t)&=\boldsymbol d_*(t)-\boldsymbol{L}_K\boldsymbol c_*(t)-r(t)\boldsymbol b_K(\boldsymbol c_*(t))\\
    &=J_K\left[\mathcal L(I-P_K^Q)\phi(t)+\mathcal U(\phi(t))-\mathcal U(P_K^Q\phi(t))\right]+[\widehat r_K(t)-r(t)]\boldsymbol b_K(\boldsymbol c_*(t)).
  \end{aligned}
\end{equation}

Let $\mathfrak Q_{M,N,K}(T)$ denote the supremum over $0\le t\le T$ of all quadrature and assembly consistency defects, including the differences between continuous and quadrature inner products, between the actual reduced matrices and their structure-preserving weak-form counterparts, and $|E_1(\phi(t))-E_{1,Q}(P_K^Q\phi(t))|$. Introduce graph norms
$$
\|v\|_{X_{\mathcal L}}=\|v\|_{L^2}+\|\mathcal Lv\|_{L^2},
\quad
\|w\|_{X_{\mathcal G}}=\|w\|_{L^2}+\|\mathcal Gw\|_{L^2},
$$
and set
\begin{equation}
  \begin{aligned}
    \mathfrak A_{M,N,K}(T):=\sup_{0\le t\le T}\bigl[&\mathfrak a_{M,N,K}^{X_{\mathcal L}}(\phi(t))+\mathfrak a_{M,N,K}^{X_{\mathcal G}}(\mu(t))
    +\mathfrak a_{M,N,K}^{L^2}(\partial_t\phi(t))\bigr].
  \end{aligned}
\end{equation}
Under $\|P_K^Q v\|_X\le C_{P,K}\|v\|_X$, $\|v_K\|_{L^2}\le C_{\mathrm{eq},K}\|v_K\|_{Q,N}$, boundedness of the exact trajectory factors in Eq.~\eqref{eq:SM nonzero spatial consistency}, local Lipschitz continuity of $\mathcal U$, and Lem.~\ref{lem:CARE truncation estimate},
\begin{equation}
  \sup_{0\le t\le T} \left(\|\boldsymbol\xi_{\mathcal G}(t)\| +\|\boldsymbol\xi_{\mathcal L,U}(t)\|+|\xi_r(t)|\right)\le C\left[\mathfrak A_{M,N,K}(T)+\mathfrak Q_{M,N,K}(T)\right].
\end{equation}

\section{Proof of Lemma~\ref{lem:Fixed-space Perturbation Stability}}
\label{sec:pf:lem:Fixed-space Perturbation Stability}
\begin{proof}
  Here and below, $D_t v^n:=(v^{n+1}-v^n)/\Delta t$. For $n\geq1$, write
  \begin{equation}
    \boldsymbol e_c^n=\boldsymbol c_*^n-\boldsymbol c^n,\qquad e_r^n=r_*^n-r^n,\qquad \boldsymbol e_d^{n+1/2}=\boldsymbol d_*^{n+1/2}-\boldsymbol d^{n+1/2},
  \end{equation}
  and set
  \begin{equation}
    \boldsymbol\beta_n=\boldsymbol b_K\bigl(\overline{\boldsymbol c}_*^{\,n+1/2}\bigr)-\boldsymbol b_K\bigl(\overline{\boldsymbol c}^{\,n+1/2}\bigr),
    \qquad \boldsymbol b=\boldsymbol b_K\bigl(\overline{\boldsymbol c}^{\,n+1/2}\bigr).
  \end{equation}
  Subtracting the numerical CARE-SAV scheme from the residual equations for the reference sequence gives
  \begin{equation}
    \begin{aligned}
      D_t\boldsymbol e_c^n&=\boldsymbol{G}_K\boldsymbol e_d^{n+1/2}+\boldsymbol\rho_{c,n},\\
      \boldsymbol e_d^{n+1/2}&=\boldsymbol{L}_K\boldsymbol e_c^{n+1/2}+e_r^{n+1/2}\boldsymbol b+r_*^{n+1/2}\boldsymbol\beta_n+\boldsymbol\rho_{d,n},\\
      D_te_r^n&=\frac12\boldsymbol b^{T} D_t\boldsymbol e_c^n+\frac12\boldsymbol\beta_n^{T} D_t\boldsymbol c_*^n+\rho_{r,n}.
    \end{aligned}
    \label{eq:SM fixed-space error equations}
  \end{equation}
  By the local Lipschitz continuity of $\boldsymbol b_K$ and the second-order extrapolation,
  \begin{equation}
    \|\boldsymbol\beta_n\|\leq \frac{L_{b,K}}{2}\left(3\|\boldsymbol e_c^n\|+\|\boldsymbol e_c^{n-1}\|\right).
    \label{eq:SM beta estimate}
  \end{equation}
  Define $\chi_n=\frac12\boldsymbol\beta_n^{T} D_t\boldsymbol c_*^n+\rho_{r,n}$. The scalar equation yields
  \begin{equation}
    e_r^{n+1}-e_r^n=\frac12\boldsymbol b^{T}(\boldsymbol e_c^{n+1}-\boldsymbol e_c^n)+\Delta t\,\chi_n,
    \label{eq:SM scalar recovery increment}
  \end{equation}
  and
  \begin{equation}
    e_r^{n+1/2}=e_r^n+\frac14\boldsymbol b^{T}(\boldsymbol e_c^{n+1}-\boldsymbol e_c^n)+\frac{\Delta t}{2}\chi_n.
    \label{eq:SM scalar recovery midpoint}
  \end{equation}
  Let $\delta\boldsymbol e_c^n=\boldsymbol e_c^{n+1}-\boldsymbol e_c^n$. Substitution gives
  \begin{equation}
    \begin{aligned}
      \boldsymbol{K}_n\delta\boldsymbol e_c^n={}&\Delta t\bigl[\boldsymbol{G}_K\boldsymbol{L}_K\boldsymbol e_c^n+\boldsymbol{G}_K\boldsymbol b\,e_r^n+\boldsymbol{G}_K r_*^{n+1/2}\boldsymbol\beta_n+\boldsymbol{G}_K\boldsymbol\rho_{d,n}+\boldsymbol\rho_{c,n}\bigr]+\frac{\Delta t^2}{2}\boldsymbol{G}_K\boldsymbol b\,\chi_n,
    \end{aligned}
    \label{eq:SM fixed-space increment equation}
  \end{equation}
  where
  \begin{equation}
    \boldsymbol{K}_n=\boldsymbol{I}_K-\frac{\Delta t}{2}\boldsymbol{G}_K\boldsymbol{L}_K-\frac{\Delta t}{4}\boldsymbol{G}_K\boldsymbol b\boldsymbol b^{T}.
    \label{eq:SM fixed-space matrix}
  \end{equation}
  Put $C_K=\frac12\|\boldsymbol{G}_K\|\,\|\boldsymbol{L}_K\|+\frac14\|\boldsymbol{G}_K\|B_b^2$. Then $\|\boldsymbol{K}_n-\boldsymbol{I}_K\|\leq C_K\Delta t$. Hence, for $C_K\Delta t\leq1/2$,
  \begin{equation}
    \|\boldsymbol{K}_n^{-1}\|\leq2.
    \label{eq:SM fixed-space inverse bound}
  \end{equation}
  Moreover,
  \begin{equation}
    |\chi_n|\leq\frac{B_t}{2}\|\boldsymbol\beta_n\|+|\rho_{r,n}|.
    \label{eq:SM chi estimate}
  \end{equation}
  Applying these bounds gives, with $C_K$ independent of $n$ and $\Delta t$,
  \begin{equation}
    \|\delta\boldsymbol e_c^n\|\leq C_K\Delta t\bigl(E_n+E_{n-1}+\mathcal K_n\bigr),
    \label{eq:SM coefficient increment bound}
  \end{equation}
  and
  \begin{equation}
    |e_r^{n+1}|\leq |e_r^n|+C_K\Delta t\bigl(E_n+E_{n-1}+\mathcal K_n\bigr).
    \label{eq:SM scalar increment bound}
  \end{equation}
  Combining these inequalities proves the stated recurrence. At the starting step, the same argument supplies the required one-step initialization estimate; equivalently, one may set $E_{-1}=E_0$. Finally, the one-step discrete Gronwall inequality completes the proof.
\end{proof}

\section{Experimental Setup}
\label{sec:SM-experimental-setup}

Unless otherwise specified, all numerical experiments are implemented in Python using double precision arithmetic. The same CARE basis-construction procedure and SAV-CN time discretization are employed throughout the experiments. Problem-dependent physical and numerical parameters are summarized in \cref{tab:baseline_parameters}. Any settings that differ from those listed in the table are stated explicitly in the corresponding subsection of the main text.

Table~\ref{tab:baseline_parameters} lists the configurations. Here, $M$ denotes the number of initially generated RBF features, $[\sigma_{\min},\sigma_{\max}]$ specifies the range of RBF widths, and $\tau_{\mathrm{QR}}$ is the tolerance used in the rank-revealing QR procedure. The quantity $K$ denotes the resulting effective dimension of the CARE space after numerically redundant feature directions have been removed. The time-step size and the total number of time steps are denoted by $\Delta t$ and $N_t$, respectively. These configurations are used as the default settings in the benchmark experiments.

\begin{table}[htbp]
\centering
\caption{Configurations of CARE-SAV for the five benchmark problems.}
\label{tab:baseline_parameters}
\small
\setlength{\tabcolsep}{5pt}
\renewcommand{\arraystretch}{1.15}
\begin{tabular}{lccccc}
\hline
Problem
& $M$
& $[\sigma_{\min},\sigma_{\max}]$
& $\tau_{\mathrm{QR}}$
& $\Delta t$
& $N_t$
\\
\hline
Allen-Cahn
& $480$
& $[0.008,0.08]$
& $1\times10^{-12}$
& $5\times10^{-5}$
& $20000$
\\
Cahn-Hilliard
& 900
& $[0.08,0.5]$
& $1\times10^{-12}$
& $6.4\times10^{-6}$
& $5000$
\\
Phase-Field Crystal
& 500
& $[0.1,0.6]$
& $1\times10^{-12}$
& $1\times10^{-3}$
& 9980
\\
Molecular Beam Epitaxy
& 1000
& $[0.1,0.6]$
& $1\times10^{-12}$
& $2.5\times10^{-4}$
& 4000
\\
Grain Growth
& 200
& $[0.1,0.6]$
& $1\times10^{-12}$
& $2\times10^{-2}$
& 5000
\\
\hline
\end{tabular}
\end{table}

\section{Parameter Sensitivity Studies}
\label{sec:SM-parameter-sensitivity}

\subsection{Number of Candidate Features}
We investigate the influence of the number of candidate features $M$ on the accuracy of CARE-SAV. Specifically, we consider $M=$ 160, 320, 480, 640 and 960, while keeping all other numerical parameters fixed. The relative $L^2$ error is reported for the Allen-Cahn and Cahn-Hilliard equations to examine how the approximation accuracy changes as the candidate feature space is enlarged.

Fig.~\ref{fig:ac_feature} and~\ref{fig:ch_feature} show the relative $L^2$ errors obtained with different numbers of candidate features for the Allen-Cahn and Cahn-Hilliard equations respectively. In both cases, enlarging the candidate feature set generally improves the approximation accuracy, while the reduction in error becomes less significant once a sufficiently rich representation is reached. This indicates that increasing $M$ enhances the expressive capability of the CARE space, but also exhibits diminishing returns beyond a moderate feature size.

\begin{figure}[htbp]
  \centering
  \includegraphics[width=1\textwidth]{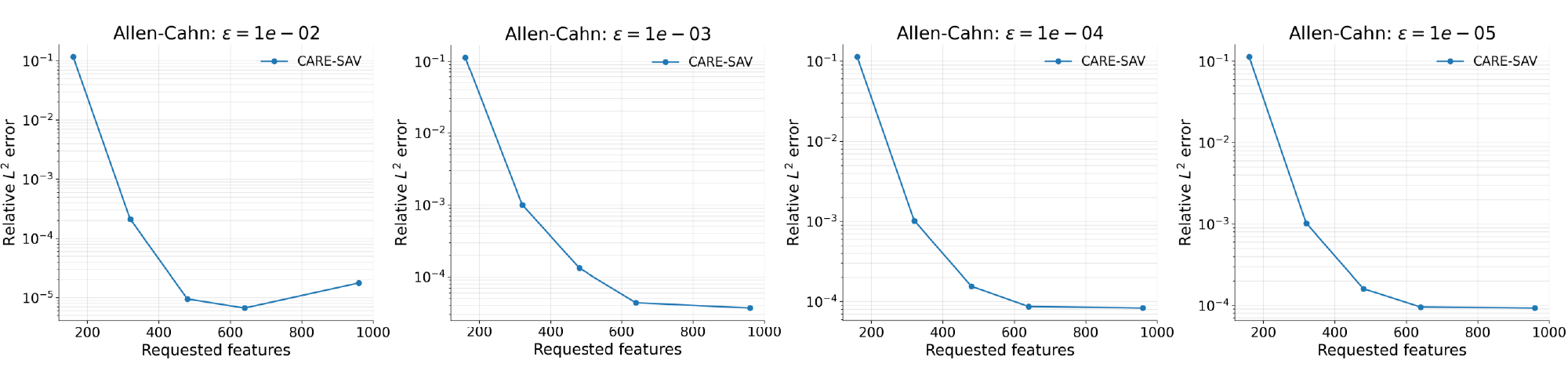}
  \caption{Relative $L^2$ error of CARE-SAV with different numbers of candidate features for the Allen-Cahn equation.}
  \label{fig:ac_feature}
\end{figure}

\begin{figure}[htbp]
  \centering
  \includegraphics[width=1\textwidth]{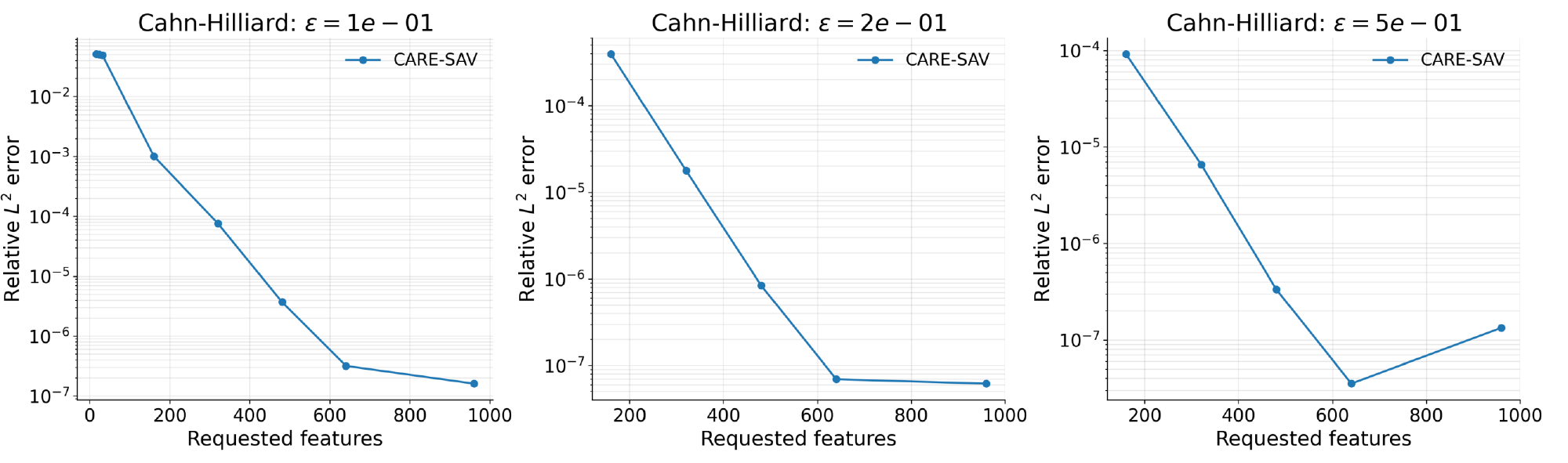}
  \caption{Relative $L^2$ error of CARE-SAV with different numbers of candidate features for the Cahn-Hilliard equation.}
  \label{fig:ch_feature}
\end{figure}

\subsection{QR Tolerance}
We examine the sensitivity of CARE-SAV to the rank-revealing QR tolerance $\tau_\text{QR}$. The number of candidate features, RBF width interval, random realization, and all discretization parameters are kept fixed, while $\tau_\text{QR}$ is varied from $10^{-8}$ to $10^{-14}$. We report the relative $L^2$ error together with the resulting effective dimension $K$ to assess the balance between feature reduction and approximation accuracy.

Fig.~\ref{fig:ac_tau} and~\ref{fig:ch_tau} show the relative $L^2$ errors obtained with different rank-revealing QR tolerances for the Allen-Cahn and Cahn-Hilliard equations respectively. The relative $L^2$ errors remain within the same order of magnitude over the tested range of QR tolerances. Although the Allen-Cahn error generally decreases as $\tau_\text{QR}$ is reduced, the Cahn-Hilliard results exhibit mild nonmonotonic variations. Nevertheless, the overall accuracy remains stable, indicating that CARE-SAV is not highly sensitive to the choice of $\tau_\text{QR}$ within the considered range and can eliminate numerically redundant feature directions without substantial loss of accuracy.

\begin{figure}[htbp]
  \centering
  \includegraphics[width=1\textwidth]{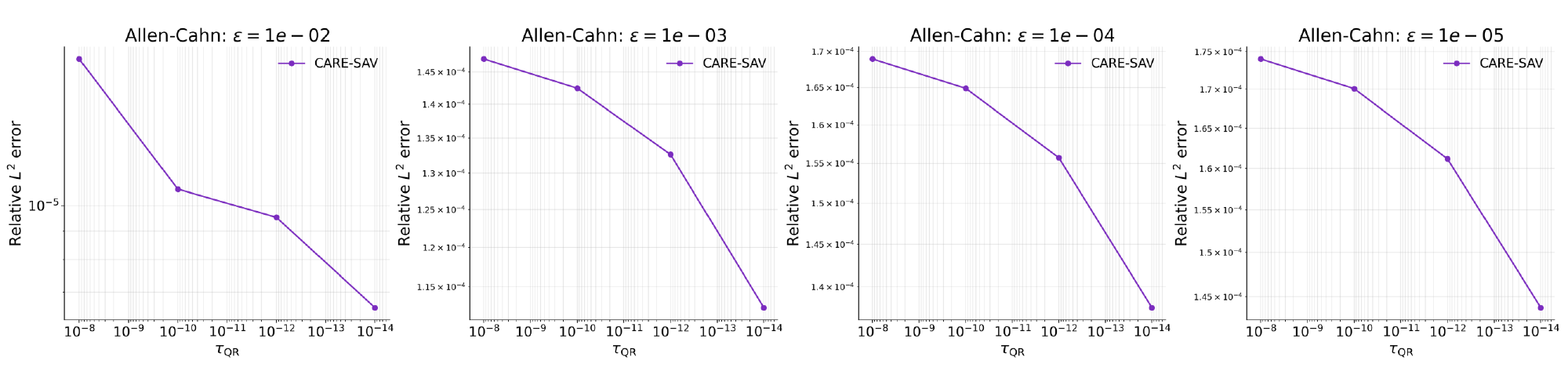}
  \caption{Relative $L^2$ error of CARE-SAV with different rank-revealing QR tolerances for the Allen-Cahn equation.}
  \label{fig:ac_tau}
\end{figure}

\begin{figure}[htbp]
  \centering
  \includegraphics[width=1\textwidth]{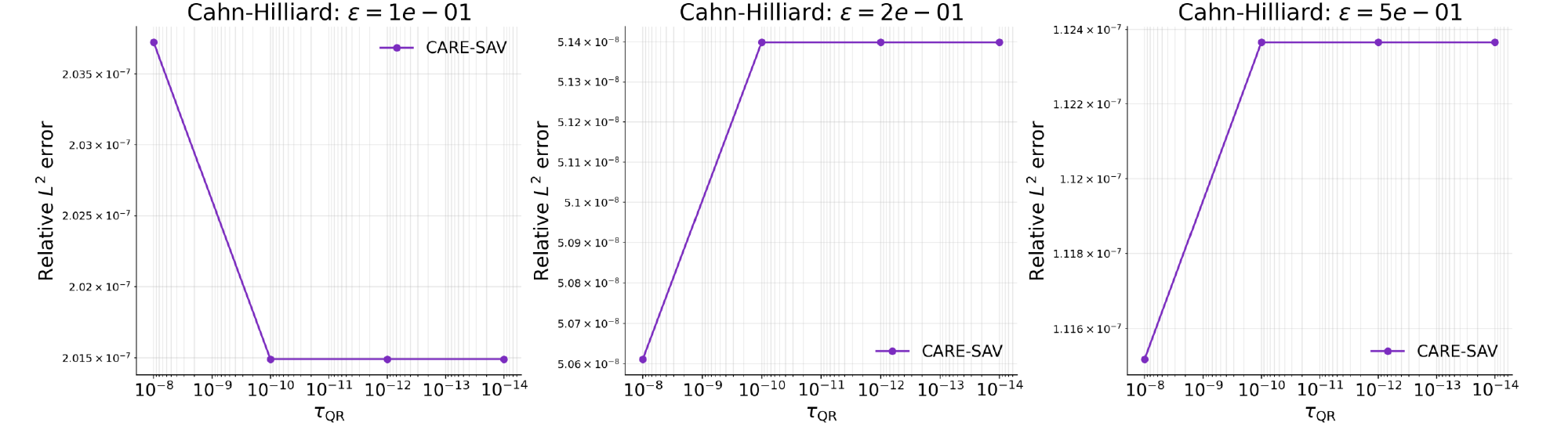}
  \caption{Relative $L^2$ error of CARE-SAV with different rank-revealing QR tolerances for the Cahn-Hilliard equation.}
  \label{fig:ch_tau}
\end{figure}

\section{Data Generator}
\label{sec:data-generator}

Reference solutions for the periodic benchmark problems are generated with a
Fourier pseudospectral discretization in space and the fourth-order
exponential time-differencing Runge-Kutta method (ETDRK4) in time
\cite{KassamTrefethen2005}. Each scalar equation, or each component of a
vector-valued system, is written in the semilinear form
\begin{equation}
  \partial_t u=\mathcal L u+\mathcal N(u),
  \label{eq:SM semilinear form}
\end{equation}
where the constant-coefficient linear operator $\mathcal L$ is treated
exactly and the nonlinear term $\mathcal N$ is evaluated pseudospectrally.
For a Fourier mode $\boldsymbol k$, let $\lambda_{\boldsymbol k}$ be the symbol of
$\mathcal L$, set $z_{\boldsymbol k}=\Delta t_{\mathrm{ref}}\lambda_{\boldsymbol k}$,
and define
\begin{equation}
  \begin{aligned}
    E_{\boldsymbol k}&=\exp(z_{\boldsymbol k}),
    &E_{2,\boldsymbol k}&=\exp(z_{\boldsymbol k}/2),\\
    Q_{\boldsymbol k}&=\frac{\Delta t_{\mathrm{ref}}}{2}
      \varphi_1(z_{\boldsymbol k}/2),\\
    f_{1,\boldsymbol k}&=\Delta t_{\mathrm{ref}}
      \left[\varphi_1(z_{\boldsymbol k})-3\varphi_2(z_{\boldsymbol k})
      +4\varphi_3(z_{\boldsymbol k})\right],
    &f_{2,\boldsymbol k}&=\Delta t_{\mathrm{ref}}
      \left[\varphi_2(z_{\boldsymbol k})-2\varphi_3(z_{\boldsymbol k})\right],\\
    f_{3,\boldsymbol k}&=\Delta t_{\mathrm{ref}}
      \left[4\varphi_3(z_{\boldsymbol k})-\varphi_2(z_{\boldsymbol k})\right],
  \end{aligned}
  \label{eq:SM ETDRK4 coefficients}
\end{equation}
with $\varphi_1(z)=(e^z-1)/z$,
$\varphi_2(z)=(e^z-1-z)/z^2$, and
$\varphi_3(z)=(e^z-1-z-z^2/2)/z^3$, extended continuously at $z=0$.

Let $\widehat{\mathcal N}(\widehat v)=
\mathcal F\mathcal N(\mathcal F^{-1}\widehat v)$, where $\mathcal F$
denotes the discrete Fourier transform. Starting from $\widehat u^n$, one
ETDRK4 step is
\begin{equation}
  \begin{aligned}
    \widehat N_n&=\widehat{\mathcal N}(\widehat u^n),\\
    \widehat a&=E_2\widehat u^n+Q\widehat N_n,
    &\widehat N_a&=\widehat{\mathcal N}(\widehat a),\\
    \widehat b&=E_2\widehat u^n+Q\widehat N_a,
    &\widehat N_b&=\widehat{\mathcal N}(\widehat b),\\
    \widehat c&=E_2\widehat a+Q(2\widehat N_b-\widehat N_n),
    &\widehat N_c&=\widehat{\mathcal N}(\widehat c),\\
    \widehat u^{n+1}&=E\widehat u^n+f_1\widehat N_n
    +2f_2(\widehat N_a+\widehat N_b)+f_3\widehat N_c.
  \end{aligned}
  \label{eq:SM ETDRK4 update}
\end{equation}
Here all products with $E$, $E_2$, $Q$, and $f_j$ are performed
mode by mode. The nonlinear terms are evaluated in physical space, and the
two-thirds rule is applied before transforming them back to Fourier space to
remove aliasing errors. The values of the coefficient functions near the
origin are evaluated by their analytic limits (or, equivalently, by the
standard contour-average implementation of ETDRK4) to avoid cancellation.

The Fourier grid and $\Delta t_{\mathrm{ref}}$ are refined until the recorded
snapshots are unchanged at the accuracy reported in the experiments. The
resulting reference fields are saved at the prescribed output times and
evaluated at the CARE collocation points by Fourier interpolation. These
snapshots provide the reference solutions used in the error measurements.

\vskip 0.1in

\noindent {\bf Data Availability} The datasets generated during and/or analysed during the current study are available from the corresponding author on request.
\vskip 0.1in
\noindent {\bf Conflict of interest} The authors declared that they have no conflict of interest.
\bibliographystyle{plain}
\nocite{*}
\bibliography{references}

@article{CALABRO2021,
title = {Extreme learning machine collocation for the numerical solution of elliptic PDEs with sharp gradients},
journal = {Computer Methods in Applied Mechanics and Engineering},
volume = {387},
pages = {114188},
year = {2021},
issn = {0045-7825},
doi = {10.1016/j.cma.2021.114188},
author = {F. Calabr{\`o} and G. Fabiani and C. Siettos},
}

@article{NASSAJIANMOJARRAD2023151,
title = {A new variable shape parameter strategy for RBF approximation using neural networks},
journal = {Computers \& Mathematics with Applications},
volume = {143},
pages = {151-168},
year = {2023},
issn = {0898-1221},
doi = {10.1016/j.camwa.2023.05.005},
author = {F. {Nassajian Mojarrad} and M. {Han Veiga} and J. S. Hesthaven and P. {\"O}ffner}
}

@misc{ren2025GFF,
      title={General Fourier Feature Physics-Informed Extreme Learning Machine (GFF-PIELM) for High-Frequency PDEs}, 
      author={F. Ren and S. Wang and P.-Z. Zhuang and H.-S. Yu and H. Yang},
      year={2025},
      eprint={2510.12293},
      archivePrefix={arXiv},
      primaryClass={cs.LG}, 
}

@article{FAN1997611,
title = {Computer simulation of grain growth using a continuum field model},
journal = {Acta Materialia},
volume = {45},
number = {2},
pages = {611-622},
year = {1997},
issn = {1359-6454},
doi = {10.1016/S1359-6454(96)00200-5},
author = {D. Fan and L.-Q. Chen},
}

@article{CrankNicolson1996,
  author  = {Crank, J. and Nicolson, P.},
  title   = {A Practical Method for Numerical Evaluation of Solutions
             of Partial Differential Equations of the Heat-Conduction Type},
  journal = {Advances in Computational Mathematics},
  volume  = {6},
  pages   = {207--226},
  year    = {1996},
  month   = dec,
  doi     = {10.1007/BF02127704}
}

@article{SHEN2018407,
title = {The scalar auxiliary variable (SAV) approach for gradient flows},
journal = {Journal of Computational Physics},
volume = {353},
pages = {407-416},
year = {2018},
issn = {0021-9991},
doi = {10.1016/j.jcp.2017.10.021},
author = {J. Shen and J. Xu and J. Yang},
}

@ARTICLE{PARKSANDBERG1991246,
  author={Park, J. and Sandberg, I. W.},
  journal={Neural Computation}, 
  title={Universal Approximation Using Radial-Basis-Function Networks}, 
  year={1991},
  volume={3},
  number={2},
  pages={246-257},
  doi={10.1162/neco.1991.3.2.246}}

@article{SHENXU20182895,
author = {Shen, J. and Xu, J.},
title = {Convergence and Error Analysis for the  Scalar Auxiliary Variable (SAV) Schemes to Gradient Flows},
journal = {SIAM Journal on Numerical Analysis},
volume = {56},
number = {5},
pages = {2895-2912},
year = {2018},
doi = {10.1137/17M1159968}
}

@article{Chen2020SAVFEM,
  author  = {Chen, H. and Mao, J. and Shen, J.},
  title   = {Optimal error estimates for the scalar auxiliary variable finite-element schemes for gradient flows},
  journal = {Numerische Mathematik},
  volume  = {145},
  pages   = {167--196},
  year    = {2020},
  doi     = {10.1007/s00211-020-01112-4},
}

@article{KassamTrefethen2005,
  author  = {Kassam, A.-K. and Trefethen, L. N.},
  title   = {Fourth-order time-stepping for stiff {PDEs}},
  journal = {SIAM Journal on Scientific Computing},
  volume  = {26},
  number  = {4},
  pages   = {1214--1233},
  year    = {2005},
  doi     = {10.1137/S1064827502410633},
}

@article{SHENDecoupled2014,
author = {Shen, J. and Yang, X.},
title = {Decoupled Energy Stable Schemes for Phase-Field Models of Two-Phase Complex Fluids},
journal = {SIAM Journal on Scientific Computing},
volume = {36},
number = {1},
pages = {B122-B145},
year = {2014},
doi = {10.1137/130921593},
}

@article{SHEN2015617,
title = {Efficient energy stable numerical schemes for a phase field moving contact line model},
journal = {Journal of Computational Physics},
volume = {284},
pages = {617-630},
year = {2015},
issn = {0021-9991},
doi = {10.1016/j.jcp.2014.12.046},
author = {J. Shen and X. Yang and H. Yu},
}

@article{YANG20171116,
title = {Linearly first- and second-order, unconditionally energy stable schemes for the phase field crystal model},
journal = {Journal of Computational Physics},
volume = {330},
pages = {1116-1134},
year = {2017},
issn = {0021-9991},
doi = {10.1016/j.jcp.2016.10.020},
author = {X. Yang and D. Han}
}

@article{Shen-Phase-Field2010,
author = {Shen, J. and Yang, X.},
title = {A Phase-Field Model and Its Numerical Approximation for Two-Phase Incompressible Flows with Different Densities and Viscosities},
journal = {SIAM Journal on Scientific Computing},
volume = {32},
number = {3},
pages = {1159-1179},
year = {2010},
doi = {10.1137/09075860X}
}

@article{GOMEZ20115310,
title = {Provably unconditionally stable, second-order time-accurate, mixed variational methods for phase-field models},
journal = {Journal of Computational Physics},
volume = {230},
number = {13},
pages = {5310-5327},
year = {2011},
issn = {0021-9991},
doi = {10.1016/j.jcp.2011.03.033},
author = {H. Gomez and T. J. R. Hughes}
}

@article{WiseFDM2009,
author = {Wise, S. M. and Wang, C. and Lowengrub, J. S.},
title = {An Energy-Stable and Convergent Finite-Difference Scheme for the Phase Field Crystal Equation},
journal = {SIAM Journal on Numerical Analysis},
volume = {47},
number = {3},
pages = {2269-2288},
year = {2009},
doi = {10.1137/080738143}
}

@article{GaoFEM2018,
author = {Gao, Y. and He, X. and Mei, L. and Yang, X.},
title = {Decoupled, Linear, and Energy Stable Finite Element Method for the Cahn--Hilliard--Navier--Stokes--Darcy Phase Field Model},
journal = {SIAM Journal on Scientific Computing},
volume = {40},
number = {1},
pages = {B110-B137},
year = {2018},
doi = {10.1137/16M1100885}
}

@article{LiShen2020SAV,
  author  = {Li, X. and Shen, J.},
  title   = {Stability and error estimates of the {SAV} Fourier-spectral method for the phase field crystal equation},
  journal = {Advances in Computational Mathematics},
  volume  = {46},
  number  = {3},
  pages   = {48},
  year    = {2020},
  doi     = {10.1007/s10444-020-09789-9}
}

@article{Eyre_1998, 
  title={Unconditionally Gradient Stable Time Marching the Cahn-Hilliard Equation}, 
  volume={529}, 
  doi={10.1557/PROC-529-39},
  journal={MRS Proceedings}, 
  author={Eyre, D. J.}, 
  year={1998}, 
  pages={39}
}

@article{BaskaranConvex2013,
author = {Baskaran, A. and Lowengrub, J. S. and Wang, C. and Wise, S. M.},
title = {Convergence Analysis of a Second Order Convex Splitting Scheme for the Modified Phase Field Crystal Equation},
journal = {SIAM Journal on Numerical Analysis},
volume = {51},
number = {5},
pages = {2851-2873},
year = {2013},
doi = {10.1137/120880677}
}

@article{Yang2009AllenCahn,
  author  = {Yang, X.},
  title   = {Error analysis of stabilized semi-implicit method of {Allen--Cahn} equation},
  journal = {Discrete and Continuous Dynamical Systems - B},
  volume  = {11},
  number  = {4},
  pages   = {1057--1070},
  year    = {2009},
  month   = {June},
  doi     = {10.3934/dcdsb.2009.11.1057},
}

@article{LiSemiImplicitFourierSpectral,
  author = {Li, D. and Qiao, Z. and Tang, T.},
  title = {Characterizing the Stabilization Size for Semi-Implicit Fourier-Spectral Method to Phase Field Equations},
  journal = {SIAM Journal on Numerical Analysis},
  volume = {54},
  number = {3},
  pages = {1653-1681},
  year = {2016},
  doi = {10.1137/140993193}
}

@article{WangYu2018CahnHilliard,
  author  = {Wang, L. and Yu, H.},
  title   = {On Efficient Second Order Stabilized Semi-implicit Schemes for the {Cahn--Hilliard} Phase-Field Equation},
  journal = {Journal of Scientific Computing},
  volume  = {77},
  pages   = {1185--1209},
  year    = {2018},
  doi     = {10.1007/s10915-018-0746-2}
}

@article{YangIEQ2017MMAMAS,
  author = {Yang, X. and Zhao, J. and Wang, Q. and Shen, J.},
  title = {Numerical approximations for a three-component Cahn-Hilliard phase-field model based on the invariant energy quadratization method},
  journal = {Mathematical Models and Methods in Applied Sciences},
  volume = {27},
  number = {11},
  pages = {1993-2030},
  year = {2017},
  doi = {10.1142/S0218202517500373}
}

@article{YANG2017104JCP,
title = {Numerical approximations for the molecular beam epitaxial growth model based on the invariant energy quadratization method},
journal = {Journal of Computational Physics},
volume = {333},
pages = {104-127},
year = {2017},
issn = {0021-9991},
doi = {10.1016/j.jcp.2016.12.025},
author = {X. Yang and J. Zhao and Q. Wang}
}

@article{YangIEQ2018SISC,
author = {Yang, X. and Yu, H.},
title = {Efficient Second Order Unconditionally Stable Schemes for a Phase Field Moving Contact Line Model Using an Invariant Energy Quadratization Approach},
journal = {SIAM Journal on Scientific Computing},
volume = {40},
number = {3},
pages = {B889-B914},
year = {2018},
doi = {10.1137/17M1125005}
}

@article{ZhaoIEQ2017IJNME,
author = {Zhao, J. and Wang, Q. and Yang, X.},
title = {Numerical approximations for a phase field dendritic crystal growth model based on the invariant energy quadratization approach},
journal = {International Journal for Numerical Methods in Engineering},
volume = {110},
number = {3},
pages = {279-300},
doi = {10.1002/nme.5372},
year = {2017}
}

@article{ChengMSAV2018SISC,
  author = {Cheng, Q. and Shen, J.},
  title = {Multiple Scalar Auxiliary Variable (MSAV) Approach and its Application to the Phase-Field Vesicle Membrane Model},
  journal = {SIAM Journal on Scientific Computing},
  volume = {40},
  number = {6},
  pages = {A3982-A4006},
  year = {2018},
  doi = {10.1137/18M1166961}
}

@article{HOU2019307,
  title = {A variant of scalar auxiliary variable approaches for gradient flows},
  journal = {Journal of Computational Physics},
  volume = {395},
  pages = {307-332},
  year = {2019},
  issn = {0021-9991},
  doi = {10.1016/j.jcp.2019.05.037},
  author = {D. Hou and M. Azaiez and C. Xu}
}

@article{HuangSAV2020SISC,
author = {Huang, F. and Shen, J. and Yang, Z.},
title = {A Highly Efficient and Accurate New Scalar Auxiliary Variable Approach for Gradient Flows},
journal = {SIAM Journal on Scientific Computing},
volume = {42},
number = {4},
pages = {A2514-A2536},
year = {2020},
doi = {10.1137/19M1298627}
}

@article{GongSAV2018SISC,
author = {Gong, Y. and Zhao, J. and Wang, Q.},
title = {Second Order Fully Discrete Energy Stable Methods on Staggered Grids for Hydrodynamic Phase Field Models of Binary Viscous Fluids},
journal = {SIAM Journal on Scientific Computing},
volume = {40},
number = {2},
pages = {B528-B553},
year = {2018},
doi = {10.1137/17M1135451}
}

@article{HanBrylevYang2017,
  author  = {Han, D. and Brylev, A. and Yang, X. and Tan, Z.},
  title   = {Numerical Analysis of Second Order, Fully Discrete Energy Stable Schemes for Phase Field Models of Two-Phase Incompressible Flows},
  journal = {Journal of Scientific Computing},
  volume  = {70},
  pages   = {965--989},
  year    = {2017},
  doi     = {10.1007/s10915-016-0279-5}
}

@article{chen2022JML,
  author = {Chen, J. and Chi, X. and E, W. and Yang, Z.},
  title = {Bridging Traditional and Machine Learning-Based Algorithms for Solving PDEs: The Random Feature Method},
  journal = {Journal of Machine Learning},
  year = {2022},
  volume = {1},
  number = {3},
  pages = {268-298},
  doi = {10.4208/jml.220726}
}

@article{SIRIGNANO20181339,
title = {DGM: A deep learning algorithm for solving partial differential equations},
journal = {Journal of Computational Physics},
volume = {375},
pages = {1339-1364},
year = {2018},
issn = {0021-9991},
doi = {10.1016/j.jcp.2018.08.029},
author = {J. Sirignano and K. Spiliopoulos}
}

@article{Belytschko1994IJNME,
author = {Belytschko, T. and Lu, Y. Y. and Gu, L.},
title = {Element-free Galerkin methods},
journal = {International Journal for Numerical Methods in Engineering},
volume = {37},
number = {2},
pages = {229-256},
doi = {10.1002/nme.1620370205},
year = {1994}
}

@article{chen2023,
  author = {Chen, J.-R. and E, W. and Luo, Y.-X.},
  title = {The Random Feature Method for Time-Dependent Problems},
  journal = {East Asian Journal on Applied Mathematics},
  year = {2023},
  volume = {13},
  number = {3},
  pages = {435-463},
  doi = {10.4208/eajam.2023-065.050423}
}

@article{CHI2024116719,
title = {The random feature method for solving interface problems},
journal = {Computer Methods in Applied Mechanics and Engineering},
volume = {420},
pages = {116719},
year = {2024},
issn = {0045-7825},
doi = {10.1016/j.cma.2023.116719},
author = {X. Chi and J. Chen and Z. Yang}
}

@article{SHANG2023107518,
title = {Randomized neural network with Petrov-Galerkin methods for solving linear and nonlinear partial differential equations},
journal = {Communications in Nonlinear Science and Numerical Simulation},
volume = {127},
pages = {107518},
year = {2023},
issn = {1007-5704},
doi = {10.1016/j.cnsns.2023.107518},
author = {Y. Shang and F. Wang and J. Sun}
}

@article{SUN2024115830,
title = {Local randomized neural networks with discontinuous Galerkin methods for partial differential equations},
journal = {Journal of Computational and Applied Mathematics},
volume = {445},
pages = {115830},
year = {2024},
issn = {0377-0427},
doi = {10.1016/j.cam.2024.115830},
author = {J. Sun and S. Dong and F. Wang}
}

@article{DWIVEDI202096,
title = {Physics Informed Extreme Learning Machine (PIELM)-A rapid method for the numerical solution of partial differential equations},
journal = {Neurocomputing},
volume = {391},
pages = {96-118},
year = {2020},
issn = {0925-2312},
doi = {10.1016/j.neucom.2019.12.099},
author = {V. Dwivedi and B. Srinivasan}
}

@article{DONG2021114129,
title = {Local extreme learning machines and domain decomposition for solving linear and nonlinear partial differential equations},
journal = {Computer Methods in Applied Mechanics and Engineering},
volume = {387},
pages = {114129},
year = {2021},
issn = {0045-7825},
doi = {10.1016/j.cma.2021.114129},
author = {S. Dong and Z. Li}
}

@article{QUAN2023697,
title = {Solving partial differential equation based on extreme learning machine},
journal = {Mathematics and Computers in Simulation},
volume = {205},
pages = {697-708},
year = {2023},
issn = {0378-4754},
doi = {10.1016/j.matcom.2022.10.018},
author = {H. D. Quan and H. T. Huynh}
}

@article{RAISSI2019686,
title = {Physics-informed neural networks: A deep learning framework for solving forward and inverse problems involving nonlinear partial differential equations},
journal = {Journal of Computational Physics},
volume = {378},
pages = {686-707},
year = {2019},
issn = {0021-9991},
doi = {10.1016/j.jcp.2018.10.045},
author = {M. Raissi and P. Perdikaris and G. E. Karniadakis}
}

@article{EYu2018DeepRitz,
  author  = {E, W. and Yu, B.},
  title   = {The Deep Ritz Method: A Deep Learning-Based Numerical Algorithm for Solving Variational Problems},
  journal = {Communications in Mathematics and Statistics},
  volume  = {6},
  pages   = {1--12},
  year    = {2018},
  doi     = {10.1007/s40304-018-0127-z}
}

@article{KHARAZMI2021113547,
title = {hp-VPINNs: Variational physics-informed neural networks with domain decomposition},
journal = {Computer Methods in Applied Mechanics and Engineering},
volume = {374},
pages = {113547},
year = {2021},
issn = {0045-7825},
doi = {10.1016/j.cma.2020.113547},
author = {E. Kharazmi and Z. Zhang and G. E. M. Karniadakis}
}

@misc{tang2026energy,
      title={Energy Dissipation Preserving Feature-based DNN Galerkin Methods for Gradient Flows}, 
      author={T. Tang and J. Yang and Y. Zhao and Q. Zhu},
      year={2026},
      eprint={2603.14029},
      archivePrefix={arXiv},
      primaryClass={math.NA},
}

@misc{wang2026pinns,
      title={When PINNs Go Wrong: Pseudo-Time Stepping Against Spurious Solutions}, 
      author={S. Wang and S. Koohy and Y. Lu and P. Perdikaris},
      year={2026},
      eprint={2604.23528},
      archivePrefix={arXiv},
      primaryClass={cs.LG},
}

@article{Celledoni2018SISC,
author = {Celledoni, E. and Eidnes, S. and Owren, B. and Ringholm, T.},
title = {Dissipative Numerical Schemes on Riemannian Manifolds with Applications to Gradient Flows},
journal = {SIAM Journal on Scientific Computing},
volume = {40},
number = {6},
pages = {A3789-A3806},
year = {2018},
doi = {10.1137/18M1190628}
}

@article{Tang2019SISC,
author = {Tang, T. and Yu, H. and Zhou, T.},
title = {On Energy Dissipation Theory and Numerical Stability for Time-Fractional Phase-Field Equations},
journal = {SIAM Journal on Scientific Computing},
volume = {41},
number = {6},
pages = {A3757-A3778},
year = {2019},
doi = {10.1137/18M1203560}
}

@article{Li2024SISC,
author = {Li, X. and Qiao, Z.},
title = {A Second-Order, Linear, \(\boldsymbol{L^\infty}\)-Convergent, and Energy Stable Scheme for the Phase Field Crystal Equation},
journal = {SIAM Journal on Scientific Computing},
volume = {46},
number = {1},
pages = {A429-A451},
year = {2024},
doi = {10.1137/23M1552164}
}
\end{document}